\documentclass[a4paper,11pt,oneside]{article}
\usepackage{graphicx} 
\usepackage[a4paper]{geometry}
\usepackage[T1]{fontenc}
\usepackage[utf8]{inputenc}
\usepackage[english]{babel}
\usepackage{microtype}
\usepackage{quoting}
\usepackage{amsthm, amsmath, amssymb, amscd, mathrsfs, mathtools} 
\usepackage{dsfont}
\usepackage{booktabs, caption, graphicx} 
\usepackage{enumitem}
\usepackage{esint} 
\usepackage{framed}
\usepackage{braket}
\usepackage{tikz-cd}
\usepackage{centernot}
\usepackage{fancyhdr}
\usepackage{emptypage}
\usepackage{csquotes}
\usepackage{leftidx}
\usepackage{tensor}
\usepackage{stmaryrd}
\usepackage{setspace}
\usepackage{latexsym}
\usepackage{comment}

\definecolor{mygreen}{RGB}{13, 110, 53}
\definecolor{Green}{RGB}{0, 128, 0}
\definecolor{navyblue}{RGB}{0, 0, 128}
\definecolor{MyMulberry}{RGB}{197, 75, 140}

\usepackage[colorlinks,linkcolor=blue, urlcolor = navyblue, citecolor = Green]{hyperref}

\usepackage[backend=biber,style=alphabetic]{biblatex}
\usepackage{settings} 

\title{Simplicity and boundary behavior of spike sequences for a superlinear problem in plasma physics}

\author{Paolo Cosentino\thanks{Department of Mathematics, University of Rome {\it ``Tor Vergata''}, Via della ricerca scientifica n.1, 00133 Roma, Italy. e-mail: cosentino@mat.uniroma2.it} \, and \, Francesco Malizia\thanks{Scuola Normale Superiore, Piazza dei Cavalieri 7, 56126 Pisa. e-mail: francesco.malizia@sns.it}}

\date{}
\begin{document}

\maketitle

{\footnotesize
		\begin{abstract}
             \noindent We prove that spike sequences related to a nonlinear problem of Grad-Shafranov type are always \emph{simple} and always converge toward \emph{interior} points of the domain. This sharpens the blow-up analysis carried out by Bartolucci-Jevnikar-Wu in \cite{bartolucci-jevnikar-wu-2025-CalcVar} and provides a converse to the existence result for spike sequences obtained by Wei in \cite{wei-2001-ProcEdinb-multiple-condensation}.

            \vspace{3ex}


	\end{abstract}}


\section{Introduction}\label{sec1}

Let us consider an open, bounded domain with smooth boundary $\Omega\subset\mathbb{R}^N$, for $N\geq3$. Given $p\in(1,p_N)$, with $p_N=\tfrac{N}{N-2}$, and $\lm>0$, we are concerned with the following problem:
\begin{equation}\label{plasma}
 \begin{cases}-\D \psi=[\alpha+\lm\psi]^p_+\qquad &\text{in}\,\,\O,
 \\[0.5ex]
 \psi=0\qquad &\text{on}\,\,\p\O,
 \\
 \int_\O[\alpha+\lm\psi]^p_+=1,
 \end{cases}
\end{equation}
where the unknown is the pair $(\alpha,\psi)\in\R\times C^{2,\beta}(\overline\O)$, for a certain $\beta\in(0,1)$.  
\\
The system \eqref{plasma} represents a simpler model of a problem involving the Grad-Shafranov operator and related to the physics of Tokamak's plasma (see \cite{stacey-plasmaphysics}). A lot of work has been done to study existence, multiplicity of solutions and other problems concerning the free boundary of the so-called \emph{plasma region} (that is, the region in which $\alpha+\lambda\psi>0$); for a detailed list of references, we direct the interested reader to \cite{bartolucci-jevnikar-wu-2025-CalcVar}. In particular, it is well known (\cite{berestycki-brezis-1980-nonlinearAnalys}) that, for any $\lambda>0$, there exists at least one solution of \eqref{plasma}. Moreover, in the case $p=1$, \eqref{plasma} is equivalent to a problem discussed by Temam in \cite{temam-1975-ARMA-plasmaequilibrium} and \cite{temam-1977-CommPDE-free-boundary-plasma}.

More recently, in \cite{bartolucci-jevnikar-2022-JD} and \cite{bartolucci-jevnikar-wu-2025-CalcVar}, the authors derived uniqueness and concentration-compactness-quantization results for solutions of \eqref{plasma}.
In particular, letting
\begin{equation*}
v:=\frac{\lambda}{|\alpha|}\psi, \qquad \mu:=\lambda\abs{\alpha}^{p-1},
\end{equation*}
the first equation in \eqref{plasma} becomes
\begin{equation}\label{eq:plasma-model-eq}
-\D v=\mu[v-1]^p_+ \hspace{0,5cm}\text{in}\,\, \O.
\end{equation}
Using this form, the authors in \cite{bartolucci-jevnikar-wu-2025-CalcVar} studied the asymptotic behavior of solutions for the following system: 
\begin{equation}\label{plasma2}
 \begin{cases} -\D v_n=\mu_n[v_n-1]^p_+\qquad &\text{in}\,\,\O,
 \\
 \,\,\,\,\mu_n\to\infty & \text{as $n\to+\infty$}, \\
 \,\,\,\,v_n\geq0,
 \end{cases}
\end{equation}
and deduce, under suitable assumptions, a vanishing-spike condensation alternative for the solutions $v_n$. In, particular, as $n\to+\infty$, the sequence $v_n$ might develop some \emph{spikes}, a behavior which was already considered in \cite{flucher-wei-1998-MathZ} and \cite{wei-2001-ProcEdinb-multiple-condensation}. \\
To give a rough idea, we refer to $v_n$ as a \emph{spike sequence} if there is at least a sequence $(x_{n})_n\subset\O$ which converges to some point $z\in\overline{\O}$, such that $v_n(x_n)\geq1+\delta$ for $\delta>0$, and, after a suitable rescaling/blow-up around $x_n$\footnote{The scaling is of type $w_n(x):= v_n(x_n+ \epsilon_nx)$, where $\epsilon_n^2:=\mu_n^{-1}$.}, one obtains a new sequence $w_n$ which converges to the unique (radially symmetric) solution $w_0$ of
\begin{align*}
\begin{cases}
    -\D w= [w-1]_+^p & \text{in}\,\,\R^N,\\
    w>0,\,\,\,\, w(x)\to0  & \text{as}\,\, |x|\to+\infty,\\
    w(0)=\underset{{x\in\R^N}}\max w(x)>1, \\
    \int_{\R^N}[w-1]_+^p<+\infty.
\end{cases}
\end{align*} 
In other words, as $n\to+\infty$, $v_n$ develops some spike-shaped shrinking profiles which converge to some point of $\overline{\Omega}$ and which, if suitably scaled, all converge to the model solution $w_0$. Moreover, the $v_n$'s converge uniformly to $0$ away from the spike concentration points, and each spike profile carries the same \acc mass'', which means that, for $R$ sufficiently large,
\[
\mu_n^{\frac{N}{2}}\int_{B_{R\epsilon_n}(x_{i,n})}[v_n-1]_+^p\to M_{p,0}:=\int_{\R^N}[w-1]_+^p=\int_{B_{R_0}(0)}[w-1]_+^p,
\]
where $\epsilon_n^2:=\mu_n^{-1}$. For a precise definition of spike set and spike sequences, we refer to Section \ref{sec:Preliminaries} and, in particular, Definition 1.7 in \cite{bartolucci-jevnikar-wu-2025-CalcVar}.

\medskip

If we also consider the Dirichlet boundary condition and the integral constraint in \eqref{plasma}, we are naturally led to the study of
\begin{equation}\label{plasma3}
 \begin{cases} -\D v_n=\mu_n[v_n-1]^p_+\qquad &\text{in}\,\,\O,
 \\[0.5ex]
 v_n=0\qquad &\text{on}\,\,\p\O,\\[0.5ex]
 \mu_n=\lambda_n|\alpha_n|^{p-1}\to\infty & \text{as $n\to+\infty$}, \\[0.5ex]
 \int_\O|\alpha_n|^p[v_n-1]^p_+=1. 
 \end{cases}
\end{equation}
In \cite[Theorem 1.13]{bartolucci-jevnikar-wu-2025-CalcVar}\footnote{Which is restated as Theorem \ref{BJWtheorem} below.}, it is shown that, assuming the following bound on the total \acc mass'' of the sequence $v_n$,
\begin{equation}
    \label{HypothesisAintroduction}
    \mu_n^{\tfrac{N}{2}}\int_{\O}[v_n-1]_+^p\,dx\leq C_0,\,\,\,\,\,\text{for some $C_0>0,$}
\end{equation}
and assuming also that either
\begin{equation}
\label{convexityintroduction}
    \O \,\,\text{is convex,}
\end{equation}
or else
\begin{equation}
    \label{HypothesisBintroduction}
    \mu_n^{\tfrac{N}{2}}\int_{\O}[v_n-1]_+^{p+1}\,dx\leq C_1,\,\,\,\,\,\text{for some $C_1>0,$}
\end{equation}
then the solutions $v_n$ of \eqref{plasma3} develop a finite number of spikes which, along a subsequence, condense into a finite spike set in $\overline{\Omega}$. Moreover, far enough from the spike set, the rescaled sequence $\mu_n^{\frac{N-2}{2}}v_n$ converges to a weighted sum of Green's functions. In the analysis of \eqref{plasma3}, the authors questioned the possibility to have in general only \emph{simple} spike points. In short, a spike point $\bar{x}\in \overline{\Omega}$ is simple when there exists only one spike profile converging to $\bar{x}$, as $n\to+\infty$, see Definition \ref{def:simplespike} (a priori one could have multiple spikes clustering at the same point in the limit). 

\medskip
The first result of this paper shows that, under the aforementioned assumptions, \emph{interior} spike points are always simple:
\begin{theorem}\label{Simplicitytheorem}
Let $v_n$ be a sequence of solutions for \eqref{plasma3} which satisfies \eqref{HypothesisAintroduction}.
Moreover, let us assume that either \eqref{convexityintroduction} or \eqref{HypothesisBintroduction} hold.
Then $v_n$ is a spike sequence and the interior spike points are simple, in the sense of Definition \ref{def:simplespike}.
\end{theorem}
\begin{remark}
    As already noticed in \cite{bartolucci-jevnikar-wu-2025-CalcVar}, the assumption \eqref{HypothesisAintroduction} is necessary to ensure that all possible singularities of $v_n$ are of spike type. Also, assumption \eqref{HypothesisBintroduction} is rather natural for this kind of problem, we refer again to \cite[Remark 1.17]{bartolucci-jevnikar-wu-2025-CalcVar}, for more details.
\end{remark}

\bigskip

Another important issue in the analysis of solutions of \eqref{plasma3} is the behavior of spike sequences at the boundary $\p\O$. In \cite{bartolucci-jevnikar-wu-2025-CalcVar}, the authors 
have ruled out the possible occurrence of boundary spike points in case of \emph{convex} domains $\O$. On the other hand, for a general domain $\Omega$, they were only able to conclude that, under the assumptions of Theorem \ref{Simplicitytheorem}, if a spike develops \acc around'' a sequence of points $(x_n)_n$, then
\begin{equation*}
    \frac{\mathrm{dist}(x_n,\p\O)}{\epsilon_n}\to+\infty \quad \text{as $n\to+\infty$},
\end{equation*}
where we remind that $\epsilon_n^2=\mu_n^{-1}$ is the scaling parameter of the spike, see Lemma 7.3 in \cite{bartolucci-jevnikar-wu-2025-CalcVar} for more details. 
In particular, there exists $R_n\to+\infty$ such that  $\epsilon_nR_n\to0^+$, as $n\to+\infty$, and $B_{\epsilon_nR_n}(x_{i,n})\Subset\O,\,\, \forall n\in\N$ and for every $(x_{i,n})_n$ converging to a boundary spike point. As a consequence, any such sequence of boundary spikes concentrates at a faster rate with respect to the rate of convergence towards the boundary, hence it must have the same profile and the same mass $M_{p,0}$ of an interior spike.\footnote{Notice that these properties are implicit in the formal definition of spike sequence given in Section \ref{sec:Preliminaries} below.}

In this paper, we prove the nonexistence of boundary spike points for \eqref{plasma3}.
\begin{theorem}\label{Boundaryspikestheorem}
Let $v_n$ be a sequence of solutions for \eqref{plasma3} which satisfies \eqref{HypothesisAintroduction}; moreover, let us assume that either \eqref{convexityintroduction} or \eqref{HypothesisBintroduction} holds.
Let $\Sigma$ be the spike set relative to $v_n$ (see Definition \ref{def:spikeset}). Then $\Sigma\subset\O$, that is, there are no spike points at the boundary of $\Omega$.
\end{theorem}

Before commenting on the proof, we make some observations.
\begin{remark}
\begin{itemize}
    \item[(a)]
    Boundary spikes have been studied at lenght in case of singularly perturbed semilinear problems with Neumann boundary conditions, starting with the seminal results in \cite{ni-takagi-1991-CPAM-shape}, \cite{ni-takagi-1993-Duke-locating}. On the other side, as far as we are concerned with Dirichlet boundary conditions, the situation seems to be more subtle since for model nonlinearities (\cite{ni-wei-1995-CPAM-spike-layer-location-dirichlet},\cite{wei-2001-ProcEdinb-multiple-condensation}) spikes generally arise at interior points. However, boundary spikes have also been constructed in \cite{li-peng-2015-CalcVar-multi-peak} for a problem similar to \eqref{plasma3}, but with an additional weight function and disregarding the integral constraint. \\
    Also, in the \acc physical'' dimension $N=2$, for $p=1$ and when considering the Grad-Shafranov operator in place of the Laplacian, Caffarelli-Friedman \cite{caffarelli-friedman-1980-Duke} proved that the spikes converge toward the boundary in the limit. Hence we see that Theorem \ref{Boundaryspikestheorem} is somewhat surprising in this context, perhaps due to the strong rigidity of the model equation \eqref{eq:plasma-model-eq}. 
    \item[(b)]
    There are many other nonlinear problems with Dirichlet boundary conditions, such as Liouville type equations (\cite{nagasaki-suzuki-1990-asymptoticanalysis},\cite{ma-wei-2001-Comment.math.Helv-liouville-pohozaev}), SU(3)-Toda systems (\cite{lin-wei-zhao-2012-Manuscripta-Todasystem}) and fourth order mean field equations (\cite{robert-wei-2008-Indiana-fourthorder}), for which it is possible to \emph{exclude} concentration on the boundary. Here the dimensions are $N=2$ (for Liouville and Toda system) and $N=4$ (for the fourth order problem in \cite{robert-wei-2008-Indiana-fourthorder}) and the singular behavior is the one typical of the equations with critical nonlinearity. In \cite{lin-wei-zhao-2012-Manuscripta-Todasystem} and \cite{robert-wei-2008-Indiana-fourthorder}, the proof of no boundary blow-up relies on a subtle application of a general Pohozaev identity at the boundary. However, their method does not provide a contradiction in our case, as the aforementioned identity is indeed verified (at least at main order). On the other hand, our approach still employs a Pohozaev-type identity but for a rescaled problem. This is more in the spirit of \cite{ma-wei-2001-Comment.math.Helv-liouville-pohozaev}, where the identity has been used for the location of blow-up points; here instead we get a contradiction because of the fact that there are no \acc stable'' spike configurations for our problem in the limit domain (which is the half-plane).\\
    Lastly, we cannot pursue a moving plane type argument, as e.g. in \cite{ma-wei-2001-Comment.math.Helv-liouville-pohozaev} or \cite{han-1991-AIHP-boundary-moving-planes} (which, in turn, relies on the arguments of \cite{defigueiredo-lions-nussbaum-1982-JmathP-movingplanes}).

    \item[(c)] Theorem \ref{Boundaryspikestheorem} solves an open problem stated in \cite{bartolucci-jevnikar-wu-2025-CalcVar} and, together with Theorem \ref{Simplicitytheorem}, refines Theorem 1.13 in \cite{bartolucci-jevnikar-wu-2025-CalcVar} (that is, Theorem \ref{BJWtheorem} below). In particular, we now know that any spike sequence for \eqref{plasma3} only develops simple interior spikes located at critical points of the Hamiltonian \eqref{Hamiltonian}. This provides a converse to the existence result of \cite{wei-2001-ProcEdinb-multiple-condensation}; we also notice that all solutions constructed in \cite{wei-2001-ProcEdinb-multiple-condensation} do satisfy the energy assumption \eqref{HypothesisBintroduction} of our theorems. 
    At last, Theorem \ref{Simplicitytheorem} is a refinement of Corollary 1.14 in \cite{bartolucci-jevnikar-wu-2025-CalcVar}, which claims that, in a convex domain, the spike set consists in a single point which coincides with the unique harmonic center of $\om$. 
    It follows from Theorem \ref{Simplicitytheorem} that the unique spike point is also simple. 
    \end{itemize}
\end{remark}

\medskip

\noindent
We now briefly comment on the proofs of Theorem \ref{Simplicitytheorem} and Theorem \ref{Boundaryspikestheorem}.

\medskip

Theorem \ref{Simplicitytheorem} is proved in Section \ref{sec:interiorspikes}. By contradiction, we suppose the existence of $l>1$ sequences of points $\{(x_{i,n})_n\}_{i=1}^l$ which converge to the same spike point $z\in\O$. Then, as it is usually done in these situations (see e.g.\cite{li-shafrir-1994-Indiana}), we define a scaling factor $\delta_n$ as the \emph{minimal distance} between any two points of these sequences, which, up to a permutation of indices, we might assume to be given by $\delta_n:=|x_{1,n}-x_{2,n}|$. We now define the scaled sequence
\[
\hat v_n(x):=v_n(x_{1,n}+\delta_n x),
\]
which satisfies 
\begin{equation*}
    \begin{cases}
    -\D \hat v_n=\hat\mu_n[\hat v_n-1]_+^p\,\, &\text{in}\, \O_n, \\[0.5ex]
    \hat v_n=0 \,\, &\text{in}\, \p\O_n,  
\end{cases}
\end{equation*}
with $\hat\mu_n:=\delta_n^2\mu$ and $\O_n:=\delta_n^{-1}(\O-x_{1,n})$. We notice that $\hat\mu_n\to+\infty$ and $\O_n\to\R^N$, as $n\to+\infty$. Also, by denoting with
\[
z_{i,n}:=\delta_n^{-1}(x_{i,n}-x_{1,n}), \qquad i\in\{1,\dots,l\},
\]
then we deduce that $z_{1,n}=0$ and $|z_{2,n}|=1$, $\forall n\in \N$. For the sake of exposition, we assume that $l=2$, and, up to a subsequence, $z_{1,n}\to z_1=0$ and $z_{1,n}\to z_2$, with $|z_2|=1$. For a small $r>0$, we denote by $\widehat\Sigma_r$ the union of the $r$-neighborhoods of all the $z_i$'s. 
Firstly, we deduce an asymptotic $C^1$ estimate for $\hat v_n$ in $B_{2R}\backslash\widehat\Sigma_r$. In particular, we obtain that the asymptotic profile of $\hat\epsilon_n^{2-N}\hat v_n$ is $C^1$-close to the sum of the fundamental solutions of the Laplace equation centered in the spikes points $z_1$ and $z_2$, cf. Lemma \ref{lemmaasymptotic}. 
Once this estimate is established, a Pohozaev-type identity implies that $z_1$, $z_2$ must satisfy
\begin{equation*}
 \frac{(z_1-z_2)}{|z_1-z_2|^N}=0,
\end{equation*}
which gives a contradiction since $z_1\not=z_2$.

\bigskip

The proof of Theorem \ref{Boundaryspikestheorem} relies on a similar approach. Again, we can argue by contradiction and assume that there exists one or more sequences of spikes converging toward some point $\bar{x}\in \partial \Omega$. Assuming for simplicity to only have one spike along a sequence $x_n\to\bar{x}$ as $n\to+\infty$, we can scale by the distance with respect to the boundary $d_n:= \mathrm{dist}(x_n,\partial \Omega)$. Then, up to a further translation and rotation, we can assume that the rescaled domains $\Omega_n$ converge toward the half-space $\R^N_-:=\{x\in \R^N\mid x_N<0\}$, while the rescaled sequence, which we call $\tilde{v}_n$, now has a spike arising at the point $z_1=(0,\dots,0,-1)\in \R^N_-$.

The next idea is to argue similarly to Theorem \ref{Simplicitytheorem}, with the difference that now we expect the asymptotic profile of $\tilde{\epsilon}_n^{2-N}\tilde{v}_n$ to be $C^1$-close to that of the Green's function for the half-space $\R^N_-$, cf. Lemma \ref{approximationlemmaboundary}. This requires additional care as we need to look at the behavior of the regular part of the Green's function when both its arguments are approaching the boundary.

Once this estimate is established, a contradiction can be reached again by means of a Pohozaev-type identity. In particular, when dealing with only one spike point $z_1$, we obtain the identity
\[
\frac{z_1-\bar z_1}{|z_1-\bar z_1|^N}=0,
\]
(here $\bar z$ is the reflection of $z$ with respect to $\R^N_-$), which implies a contradiction. In the general case, we conclude by means of a similar identity, see Section \ref{subsec:boundarypohozaev}.

\bigskip

We end the Introduction with a brief description of the contents in the next sections. In Section \ref{sec:Preliminaries}, we provide the precise definitions of spike points and spike sequences and we further recall some theorems (obtained in \cite{bartolucci-jevnikar-wu-2025-CalcVar}) on the analysis of \eqref{plasma3}. In Section \ref{sec:interiorspikes}, we analyze the formation of interior spikes and prove Theorem \ref{Simplicitytheorem}. Finally, the analysis of boundary spikes sequences, as well as the proof of Theorem \ref{Boundaryspikestheorem}, is carried out in Section \ref{sec:Boundaryspikes}.

\section{Preliminaries}\label{sec:Preliminaries}

In this Section, we recall all the terminology and definitions about spike points and spike sequences which is needed throughout the paper, as well as the precise statement of the condensation-vanishing Theorem by \cite{bartolucci-jevnikar-wu-2025-CalcVar}, which is the starting point of our analysis. 

\begin{definition}[\cite{bartolucci-jevnikar-wu-2025-CalcVar}]\label{def:spikeset}
Let us set $\epsilon_n^2=\mu_n^{-1}$. The \emph{spike set} relative to a sequence of solutions of \eqref{plasma2} is the largest set $\Sigma\subset\overline{\O}$ such that, for any $z\in\Sigma$, there exists $x_n\to z$, $\sigma_z>0$ and $M_z>0$ with the following properties:
 \begin{itemize}
  \item $v_n(x_n)\geq 1+\sigma_z$, for every $n\in \N$;
  \item $\epsilon_n=o(\mathrm{dist}(x_n,\p\O))$, as $n\to+\infty$;
  \item $\mu_n^{\frac{N}{2}}\int_{B_r(z)\cap\O}\vnpp\leq M_z$, $\forall \,r>0$.
 \end{itemize}
 Any point $z\in \Sigma$ is said to be a \emph{spike point}.
\end{definition}

\begin{definition}[essentially \cite{bartolucci-jevnikar-wu-2025-CalcVar}]\label{spikessequencesdef}
    A sequence $(v_n)_n$ of solutions of \eqref{plasma2} with nonempty spike set $\Sigma=\{z_1,\dots,z_l\}\subset\overline\O$ is said to be a \emph{spike sequence} (relative to $\Sigma$) if:
\begin{itemize}
 \item there exist (strictly) positive integers $m_1,\dots,m_l$ and sequences $(z_{i,n})_n$, $i=1,\dots,Z:=\sum_{k=1}^l m_k$, with $z_{i,n}\to z_j$, as $n\to+\infty$, whenever $\sum_{k=0}^{j-1} m_k< i< \sum_{k=0}^{j} m_k$ (here $m_0:=0$);
 \item there exist $R_n\to+\infty$ as $n\to+\infty$, a natural number $\overline n\in\N$ and positive constants $\tilde C>0,$ $\sigma>0$, such that the following hold: 
 \begin{itemize}
  \item [(i)] $\epsilon_n R_n\to 0$, as $n\to+\infty$;
  \item [(ii)] it holds $$ B_{4R_n}(0)\Subset\O_{i,n}:=\tfrac{\O-z_{i,n}}{\epsilon_n}, \,\,\,\, \forall i\in  \{1,\dots,Z\};$$
  \item [(iii)] if $Z\geq2$, then
  \[
   B_{2\epsilon_nR_n}(z_{j,n})\cap B_{2\epsilon_nR_n}(z_{k,n})=\emptyset,\,\,\,\, \forall j\neq k,
  \]
and
\[
 v_n(z_{i,n})=\underset{\overline\O\backslash \bigcup_{j=1,j\neq i}^Z B_{2\epsilon_n R_n}(z_{j,n})\}}\max \,v_n\geq 1+\sigma,\,\,\,\, \forall i\in\{1,\dots,Z\};
\]
\item [(iv)] for each $i\in\{1,\dots,Z\}$, the rescaled functions
\begin{equation}\label{eq:scaled-spike-def}
 w_{i,n}(y):=v_n(\epsilon_n y+z_{i,n}),\,\,\,\, y\in \O_{i,n},
\end{equation}
satisfy
\[
 \|w_{i,n}-w_0\|_{C^{2}(B_{2R_n}(0))}\to0 \quad \text{as $n\to+\infty$};
\]
moreover, the ``plasma region'' $\O_{n,+}:=\{x\in\O \mid v_n>1\}$ satisfies 
\begin{equation}
    \label{positivityregion}
\Omega_{n,+}\Subset \bigcup_{i=1}^Z B_{2\epsilon_nR_0}(z_{i,n}) \qquad \forall n\geq \bar{n},
\end{equation}
see \eqref{eq:w_0-def}, \eqref{eq:M_p,0-def} below for the definitions of $w_0$ and $R_0$;
\item [(v)] if we set $\Sigma_r:=\underset{j=1,\dots,l}\bigcup B_{r}(z_j)$, for some $r>0$, and $\tilde\Sigma_n:=\underset{i=1,\dots,Z}\bigcup B_{2\epsilon_n R_n}(z_{i,n})$, then there exists a positive constant $C_r$ (depending upon $r$) such that
\[
 0\leq \underset{\overline\O\backslash\Sigma_r}\max\, v_n\leq C_r \epsilon_n^{N-2},\hspace{1,5cm}0\leq \underset{\overline\O\backslash\tilde\Sigma_n}\max\, v_n\leq \frac{\tilde C} {R_n^{N-2}},\,\,\,\,\forall n\geq\overline n;
\]
\item [(vi)] each sequence $(z_{i,n})_n $ carries a fixed amount of \acc mass'' $M_{p,0}$:
\[
 \underset{n\to\infty}\lim \underset{B_{\epsilon_n R_n}(z_{i,n})}\int \mu_n^{\frac{N}{2}}\vnpp=\underset{n\to\infty}\lim \underset{B_{R_n}(0)}\int \mu_n^{\frac{N}{2}}[w_{i,n}-1]_+^p=M_{p,0}.
 \]
Moreover, for any $\varphi\in C^0(\overline{\O})$,
\[\underset{n\to\infty}\lim \underset{\O}\int \mu_n^{\frac{N}{2}}\vnpp \varphi\,dx=M_{p,0}\sum_{j=1}^l m_j\varphi(z_j).
\]
 \end{itemize} 
\end{itemize}
We call a \emph{spike} each local shrinking profile in $B_{\epsilon_n R_n}(z_{i,n})$ which, after scaling as in \eqref{eq:scaled-spike-def}, is approximating $w_0$.
\end{definition}

\begin{remark}
    Definition \ref{spikessequencesdef} is slightly less general than Definition 1.8 in \cite{bartolucci-jevnikar-wu-2025-CalcVar}; however, in our situation (i.e. under the assumptions of Theorem \ref{BJWtheorem}), those two definitions are essentially equivalent.
\end{remark}

\begin{definition}\label{def:simplespike}
     In the above definition, we call $m_j$ the \emph{multiplicity} of the spike point $z_j$, $\forall j=1,\dots,l$. In other words, a spike point $z$ has multiplicity $m$ whenever there are exactly $m$ spikes clustering at $z$, as $n\to+\infty$. In particular, when $m=1$, we say that $z$ is a \emph{simple} spike point.
\end{definition}

\begin{remark}
Let us define $w_0$ as 
\begin{align}\label{eq:w_0-def}
    w_0(x)=\begin{cases}
       1+R_0^{\frac{2}{1-p}}u(R_0^{-1}|x|) &\text{if}\,\,|x|\in[0,R_0],\\[0.7ex]
        \left(\frac{R_0}{|x|}\right)^{N-2}&\text{if}\,\, |x|\in[R_0,+\infty),
    \end{cases}
\end{align}
where $u$ is the unique radial (\cite{gidas-ni-nirenberg-1979}) solution of
\begin{equation*}
\begin{cases}
    -\D u= u^p\,\,\,\,\, &\text{in}\,\,B_1(0),\\
    u>0&\text{in}\,\,B_1(0), \\
    u=0&\text{on}\,\,\p B_1(0),
\end{cases}
\end{equation*}
and let
\begin{equation*}
    R_0=\left(\frac{-u'(1)}{N-2}\right)^{\frac{p-1}{2}}.
\end{equation*}
It is easy to see that $w_0$ is the unique solution of 
\begin{align*}
\begin{cases}
    -\D w= [w-1]_+^p\,\,\,\,\,\,\,\,\,\text{in}\,\,\R^N,\\
    w>0,\,\,\,\, w(x)\to0  \,\,\,\,\text{as}\,\, |x|\to+\infty,\\
    w(0)=\underset{{x\in\R^N}}\max w(x)>1, \\
    \int_{\R^N}[w-1]_+^p<+\infty.
\end{cases}
\end{align*}
Also, we define the following quantity:
\begin{equation}\label{eq:M_p,0-def}
    M_{p,0}:=\int_{\R^N}[w_0-1]_+^p=\int_{B_{R_0}(0)}[w_0-1]_+^p.
\end{equation}
We direct the interested reader to Lemma 6.2 and Remark 6.5 in \cite{bartolucci-jevnikar-wu-2025-CalcVar} for further comments on $w_0$ and $M_{p,0}$.
\end{remark}

Finally, we state the condensation-vanishing theorem of \cite{bartolucci-jevnikar-wu-2025-CalcVar}:

\begin{theorem}[\cite{bartolucci-jevnikar-wu-2025-CalcVar}, Theorem 1.11]
    Let $v_n$ be a sequence of solutions of \eqref{plasma2} with spike set $\Sigma\subset\bar\O$ such that \eqref{HypothesisAintroduction} holds and
    \begin{equation*}
        \mu_n^{\frac{N}{2}}\int_{\O}v_n^{\frac{2N}{N-2}}\,dx\leq D.
    \end{equation*}
for some positive constant $D$. Then,
\begin{itemize}
    \item (Vanishing) $[v_n-1]_+\to0$ locally uniformly in $\O$,in which case for any open and relatively compact subset $\O_0\Subset\O$ there exists $n_0\in\N$ and $C>0$, depending on $\O_0$, such that $[v_n-1]_+=0$ in $\O_0$ and in particular
    \[
    \|v_n\|_{L^\infty(\O_0)}\leq C\epsilon_n^{\frac{2N}{N-2}},\,\,\,\,\forall n>n_0,
    \]
\end{itemize}
or
\begin{itemize}
    \item (Spikes-Condensation) up to a subsequence the interior spikes set relative to $v_n$ is nonempty and finite, that is $\Sigma_0:=\Sigma\cap\O=\{z_1,\cdots, z_l\}$, for some $m\geq1$ and for any open and relatively compact set $\O'$ such that $\Sigma_0\subset\O'\Subset\O$, $\{v_n|_{\O'}\}$ is a spike sequence (relative to $\Sigma_0$). 
\end{itemize}
\end{theorem}

\noindent However, if $v_n$ is a solution of \eqref{plasma3}, then we have a better understanding of the internal and boundary behavior of the spike condensating sequences. \\
Firstly, we define the $\vec k-$\emph{Kirchhoff-Routh Hamiltonian}
\begin{equation}
\label{Hamiltonian}
\mathcal{H}(x_1,\dots,x_m;\vec k):=\sum_{i=1}^m k_j^2 H(x_i,x_i)+\sum_{i\neq j=1}^m k_ik_jG(x_i,x_j),
\end{equation}
where $\vec k=(k_1,\dots,k_m)\in\R^m,\,\,k_i\neq0$, $G(x,y)=G_\O(x,y)$ is the Green function with Dirichlet boundary condition for $\O$, that is
\begin{equation}\label{Greendef}
\begin{cases}
-\D_x G(x,y)=\delta_y(x)\,\, &x\in\O, \\
G(x,y)=0\,\, &x\in\p\O,
\end{cases}
\end{equation}
and $H$ is the regular part of the Green function
\begin{equation}\label{Robindef}
H(x,y)=H_{\Omega}(x,y):=G(x,y)-\frac{C_N}{|x-y|^{N-2}},
\end{equation}
where $C_N=\tfrac{1}{N(N-2)\omega_N}$ and $\omega_N$ is the volume of the $N$-dimensional unit ball. \\ In \cite{bartolucci-jevnikar-wu-2025-CalcVar}, the authors proved the following result
\begin{theorem}[\cite{bartolucci-jevnikar-wu-2025-CalcVar}, Theorem 1.13]\label{BJWtheorem}
Let $v_n$ be a sequence of solutions for \eqref{plasma3} which satisfies  \eqref{HypothesisAintroduction}.
Moreover, let us assume that either \eqref{convexityintroduction} or \eqref{HypothesisBintroduction} holds.
Then, up to a subsequence, $v_n$ is a spike sequence relative to $\Sigma$, for some $\Sigma=\{z_1,\cdots,z_l\}$, $l\geq1$. Moreover, by setting $\Sigma_r=\bigcup_{i=1,\dots,m}B_r(z_j)$ for some $r>0$, then
\begin{equation*}
\label{Greenapproximation}
v_n(x)=\epsilon_n^{N-2}\sum_{i=1}^{Z} M_{p,0}G(x,z_{i,n})+ o(\epsilon_n^{N-2}),\,\, \forall x\in \O\backslash \Sigma_r,
\end{equation*}
for any $n$ large enough and we have the following ``mass quantization identity":
\begin{equation*}
\lim_{n\to\infty} \bigg(\frac{\lambda_n}{|\alpha_n|^{1-\frac{p}{p_N}}}\bigg)^\frac{N}{2}=\lim_{n\to\infty} \underset{\O}\int \mu_n^{\frac{N}{2}}\vnpp=ZM_{p,0}.
\end{equation*} 
At last, if $\Sigma\subset\O$, then, up to a permutation, $(z_1,\dots,z_m)$ is a critical point of the $\vec k-$Kirchhoff-Routh Hamiltonian \eqref{Hamiltonian}, with $\vec k=(k_1,\dots,k_m)$.
\end{theorem}

\begin{remark}
 We observe that the thesis of Theorem \ref{Simplicitytheorem} and Theorem \ref{Boundaryspikestheorem} imply that $\Sigma\subset\O$ and $(z_1,\dots,z_m)$ is a critical point of the $\vec k-$Kirchhoff-Routh Hamiltonian \eqref{Hamiltonian}, with $\vec k=(k_1,\dots,k_m)$  satisfying
$$k_i=1,\,\,\,\, \forall i\in\{1,\dots,m\}.$$
\end{remark}


\section{Simplicity of interior spike points}\label{sec:interiorspikes}

\noindent
In this section we will prove Theorem \ref{Simplicitytheorem}. The fact that $v_n$ is a spike sequence is a consequence of Theorem \ref{BJWtheorem}. Let us fix an interior spike point $\overline z\in\O\cap\Sigma$. \\
We prove the result by contradiction, so we start assuming that there exist $l>1$ sequences $x_{1,n},\dots,x_{l,n}$, of local maxima corresponding to $l$ spikes sequences clustering at $\overline z$.\\
Without loss of generality, we assume 
\[
\delta_n:=|x_{1,n}
-x_{2,n}|=\underset{i\neq j}\min\,{|x_{i,n}
-x_{j,n}|}.
\]
Let us define 
\begin{align*}
\hat v_n(x):=v_n(x_{1,n}+\delta_n x),
\end{align*}
and let 
\begin{equation*}
\hat {\mu}_n:=\delta_n^2\mu_n.
\end{equation*}
Then $\hat v_n$ satisfies
\begin{equation}
    \label{hatvneq}
    \begin{cases}
    -\D \hat v_n=\hat\mu_n[\hat v_n-1]_+^p\,\, &\text{in}\, \O_n, \\[0.5ex]
    \hat v_n=0 \,\, &\text{in}\, \p\O_n,  
\end{cases}
\end{equation}
where $\O_n:=\delta_n^{-1}(\O-x_{1,n})$. By reminding that $\epsilon_n^2=\mu_n^{-1}$ and using the fact that $v_n$ is a spike sequence according to Theorem \ref{BJWtheorem}, then $(iii)$ in the definition \ref{spikessequencesdef} holds and, as a consequence,
\begin{equation*}
    \sqrt{\hat{\mu}_n}=\frac{\sqrt{\hat{\mu}_n}}{\sqrt{\mu_n}\epsilon_n}=\frac{\delta_n}{\epsilon_n}>R_n \to+\infty \quad \text{as $n\to+\infty$},
\end{equation*}
hence $\hat{\mu}_n\to+\infty$ as $n\to+\infty$.
Moreover, assumption \eqref{HypothesisAintroduction} implies by a change of variable that
\[
\hat {\mu}_n^{\frac{N}{2}}\int_{\O_n}[\hat v_n-1]_+^p\,dx\leq C_0.
\]
We notice also that, by denoting with 
\[
z_{i,n}:=\delta_n^{-1}(x_{i,n}-x_{1,n}), \,\,\,\,\,\text{for}\,\,i\in\{1,\dots,l\},
\] 
we have that $z_{1,n}=0$ and $|z_{2,n}|=1$ $\forall n$, whence $z_{1,n}\to0=:z_1$ and, up to a subsequence, $z_{2,n}\to z_2\in \S^{N-1}$ as $n\to+\infty$.\\
We do not know a priori the behavior of the other sequences; however, we can assume without loss of generality that there exist $R>0$ and $2\leq k\leq l$ such that
\[
|z_{i,n}|\leq R, \,\,\, \forall n\in\N, \quad\text{for}\,\,i\in\{1,\dots, k\},
\]
while 
\[
|z_{i,n}|\to+\infty \,\,\,\text{as $n\to+\infty$},\quad\text{for}\,\,i\in\{k+1,\dots,l\}.
\]
Moreover, for any $i\in\{1,\dots, k\}$, there exists $z_i\in B_R(0)$ such that, up to a subsequence, $z_{i,n}\to z_i$, as $n\to+\infty$.
Let us define 
\begin{equation*}
\widehat \Sigma_r:=\bigcup_{i=1}^k B_r(z_i),
\end{equation*}
for $r>0$ small. We now proceed in two steps:
    \begin{itemize}
        \item [(1)] First, we derive an asymptotic $C^1$ estimate for $\hat{v}_n$ in $B_{2R}\backslash \widehat{\Sigma}_r$.
        \item [(2)] Secondly, we use a Pohozaev-type identity together with the aforementioned estimate in order to get a contradiction.
    \end{itemize}

\subsection{Pointwise estimates for $\hat{v}_n$}

Let $G_\Omega$ be the Green's function of the Laplacian in $\Omega$ defined in \eqref{Greendef} and $H_\Omega$ be its regular part given by \eqref{Robindef}.
Let us consider $G_n(x,y):=G_{\O_n}(x,y)$, the Green function of the Laplacian in $\O_n$.
It is straightforward to verify that 
\begin{equation}\label{eq:scaledgreen-int}
G_n(x,y)=\delta_n^{N-2} G_{\O}(x_{1,n}+\delta_n x, x_{1,n}+\delta_n y)= \tfrac{C_N}{|x-y|^{N-2}}+ H_n(x,y),
\end{equation}
where 
\begin{equation}
    \label{Hndefinition}
H_n(x,y):=\delta_n^{N-2} H_{\O}(x_{1,n}+\delta_n x,x_{1,n}+\delta_n y)
\end{equation}
is the regular part of the Green function $G_n$ in $\O_n$. \\
Now, we want to estimate $H_n(x,y)$ in a suitable region. Let $\tau>0$ small enough such that $B_\tau(\bar z)\Subset\O$ and $B_{2\tau}(\bar{z})\cap \Sigma=\{\bar{z}\}$; then there exists $L=L(\O,\tau)>0$ such that
$$
\begin{cases}
    |H_\Omega(x,y)|\leq L,    \\
    |\nabla_xH_\Omega(x,y)|\leq L, & \text{in $\overline{\Omega}\times \bar{B}_\tau(\bar{z})$}. 
\end{cases}
$$
By definition of $H_n$, we have that, for $n$ large enough,
\begin{equation*}
\|H_n\|_{C^0(\overline{\O}_n\times \bar{B}_{\frac{\tau}{\delta_n}}(\bar z))}\leq L\delta_n^{N-2},
\end{equation*}
\begin{equation}
\label{Robinscaledderivative}
\|\nabla_xH_n\|_{C^0(\overline{\O}_n\times \bar{B}_{\frac{\tau}{\delta_n}}(\bar z))}\leq L\delta_n^{N-1}.
\end{equation}
In particular, by using the symmetry in $x$ and $y$ of $H_n$, we deduce that
\begin{equation}
\label{Robinscaled}
\|H_n\|_{C^1(\overline{\O}_n\times \bar{B}_{\frac{\tau}{\delta_n}})}=
\|H_n\|_{C^1(\bar{B}_{\frac{\tau}{\delta_n}}\times \overline{\O}_n)}\leq L\delta_n^{N-2}.
\end{equation}
We are now able to estimate $\hat{v}_n$ as follows:
\begin{lemma}\label{lemmaasymptotic}
  Let $\hat\epsilon_n^2:=\hat\mu_n^{-1}$ and $x\in \bar B_{2R}(0)\backslash \widehat\Sigma_r$. Then
\begin{equation*}
    \hat\epsilon_n^{2-N}\hat v_n(x)=\sum_{i=1}^kM_{p,0} \frac{C_N}{|x-z_{i,n}|^{N-2}}+R_n(x),
\end{equation*}
where the remainders $R_n$ satisfy
\begin{equation*}
    \|R_n\|_{L^\infty(\bar{B}_{2R}(0)\backslash \widehat\Sigma_r)}=o_n(1)
\end{equation*}
and 
\begin{equation*}
    \|\nabla R_n\|_{L^\infty(\bar{B}_{2R}(0)\backslash \widehat\Sigma_r)}=o_n(1)
\end{equation*}
as $n\to+\infty$.

\end{lemma}

\begin{proof}
By using \eqref{hatvneq}, the Green representation formula for $\hat v_n$ in $\O_n$ and the definition \eqref{eq:scaledgreen-int} of $G_n$, we have that
\begin{align}\nonumber
   R_n(x)&:=\hat\epsilon_n^{2-N}\hat v_n(x)-\sum_{i=1}^kM_{p,0} \tfrac{C_N}{|x-z_{i,n}|^{N-2}} \\ \nonumber
    &=\hat\epsilon_n^{2-N}\hat v_n(x)-\sum_{i=1}^k M_{p,0}G_n(x,z_{i,n})+\sum_{i=1}^{k} M_{p,0}H_n(x,z_{i,n}) \\
    &=\int_{\O_n}G_n(x,y)\hat\mu_n^{\frac{N}{2}}[\hat v_n(y)-1]_+^p\,dy-\sum_{i=1}^k M_{p,0}G_n(x,z_{i,n})+\sum_{i=1}^{k} M_{p,0}H_n(x,z_{i,n}).\label{expression}
\end{align}
Now, by using the fact that $ v_n$ is a spike sequence, we remark that $[{v}_n-1]_+^p$ is positive only in the \acc spike regions'' of ${v}_n$. More precisely, by \eqref{positivityregion}, one has
\[
[{v}_n-1]_+^p=0 \,\,\,\,\,\,\,\text{in}\,\,\O\backslash\bigcup_{i=1}^{l+\bar m}B_{2\epsilon_nR_0}(x_{i,n}),
\]
where $l$ is the total number of spikes that converge to $\bar z$ and $\bar m\geq0$ counts the number of spike profiles of $v_n$ which may converge to other spike points in $\bar \O$. From the latter property, by definition of $\hat v_n$, it is straightforward to deduce that
\begin{equation}
\label{vanishingoutsidespikes}
[{\hat v}_n-1]_+^p=0 \,\,\,\,\,\,\,\text{in}\,\,\O\backslash\bigcup_{i=1}^{l+\bar m}B_{2\hat\epsilon_nR_0}(z_{i,n})
\end{equation}
and then,
\begin{align*}\nonumber
\int_{\O_n}&G_n(x,y)\hat\mu_n^{\frac{N}{2}}[\hat v_n(y)-1]_+^p\,dy= \\
&=\sum_{i=1}^k\int_{B_{2R_0\hat\epsilon_n(z_{i,n})}}G_n(x,y)\hat\mu_n^{\frac{N}{2}}[\hat v_n(y)-1]_+^p\,dy + \sum_{j=k+1}^{l+\bar m}\int_{B_{2R_0\hat\epsilon_n(z_{j,n})}}G_n(x,y)\hat\mu_n^{\frac{N}{2}}[\hat v_n(y)-1]_+^p\,dy.
\end{align*}

We split the calculation of \eqref{expression} into three parts as follows:
\begin{equation}\label{tripartition}
    \int_{\O_n}G_n(x,y)\hat\mu_n^{\frac{N}{2}}[\hat v_n(y)-1]_+^p\,dy-\sum_{i=1}^k M_{p,0}G_n(x,z_{i,n})+\sum_{i=1}^{k} M_{p,0}H_n(x,z_{i,n})=: \alpha_n(x)+\beta_n(x)+\gamma_n(x),
\end{equation}
where
\begin{equation}
    \label{I}
\alpha_n(x)=\sum_{i=1}^k\bigg(\int_{B_{2R_0\hat\epsilon_n}(z_{i,n})} G_n(x,y)\hat\mu_n^{\frac{N}{2}}[\hat v_n(y)-1]_+^p\,dy \,\,-M_{p,0} G_n(x,z_{i,n})\bigg),
\end{equation}
\begin{equation}
    \label{II}
\beta_n(x)=\sum_{i=k+1}^{l+\bar m}\int_{B_{2R_0\hat\epsilon_n}(z_{i,n})} G_n(x,y)\hat\mu_n^{\frac{N}{2}}[\hat v_n(y)-1]_+^p\,dy
\end{equation}
and
\begin{equation}
    \label{III}
\gamma_n(x)=\sum_{i=1}^{k} M_{p,0}H_n(x,z_{i,n}).
\end{equation}
First of all, we consider \eqref{III} and easily deduce from \eqref{Hndefinition} that
\[
\max_{x\in \bar{B}_{2R}}\abs{\gamma_n(x)}\to 0 \quad \text{as $n\to\infty$}.
\]
Now, we want to estimate \eqref{I}. We remind that $\{z_{i,n}\}_{i=1}^k$ are converging to spike points inside the ball of radius $2R$. For any $i\in\{1,\dots,k\}$, we see that
\begin{align*}
    \int_{B_{2R_0\hat\epsilon_n}(z_{i,n})} &G_n(x,y)\hat\mu_n^{\frac{N}{2}}[\hat v_n(y)-1]_+^p\,dy \,-M_{p,0} G_n(x,z_{i,n})\\
    &=\int_{B_{2R_0}(0)} G_n(x,\hat\epsilon_n z+z_{i,n})[\hat v_n(\hat\epsilon_n z+z_{i,n})-1]_+^p\,dz \,-M_{p,0} G_n(x,z_{i,n})\\
    &=:\int_{B_{2R_0}(0)} G_n(x,\hat\epsilon_n z+z_{i,n})[w_{i,n}(z)-1]_+^p\,dz \,-M_{p,0} G_n(x,z_{i,n}),
\end{align*}
where $w_{i,n}\to w_0$ in $C^{2}_{loc}$ as in definition \ref{spikessequencesdef} (iv). Here, we have used the fact that
\begin{equation*}
\hat{v}_n (\hat{\epsilon}_n z+z_{i,n})=v_n (x_{1,n}+\delta_n (\hat{\epsilon}_n 
z+z_{i,n}))=v_n(x_{i,n}+\epsilon_n 
 z)=w_{i,n}(z).
\end{equation*}
Thus, by using the definition of $M_{p,0}$ \eqref{eq:M_p,0-def}, we have that
\begin{align*}
    \bigg|\int_{B_{2R_0}(0)} &G_n(x,\hat\epsilon_n z+z_{i,n})[w_{i,n}(z)-1]_+^p\,dz \,-M_{p,0} G_n(x,z_{i,n})\bigg|\\
    &\leq \underbrace{\int_{B_{2R_0}(0)} \big|G_n(x,\hat\epsilon_n z+z_{i,n})-G_n(x,z_{i,n})\big|[w_0(z)-1]_+^p\,dz}_{(A_n)} +\\
    &\hspace{2cm}+\underbrace{\int_{B_{2R_0}(0)} G_n(x,\hat\epsilon_n z+z_{i,n})\big|[w_{i,n}-1]_+^p-[w_0(z)-1]_+^p\big|\,dz}_{(B_n)}.\\
\end{align*}
To estimate $(A_n)$, we remind that, being $x\in B_{2R}\backslash\widehat\Sigma_r$, then $|x-z_{i,n}|\geq\tfrac{r}{2}$ for $n$ large enough and for every $i\in\{1,\dots,k\}$.
Therefore, by recalling that $\abs{z}\leq 2R_0$ and using \eqref{Robinscaled}, we deduce that
\begin{align*}
    \Big|G_n(x,\hat\epsilon_n z&+z_{i,n})-G_n(x,z_{i,n})\Big|=\\
    &=\Big|\tfrac{C_N}{|x-z_{i,n}-\hat\epsilon_n z|^{N-2}}-\tfrac{C_N}{|x-z_{i,n}|^{N-2}}+ H_n(x,\hat\epsilon_n z+z_{i,n})-H_n(x,z_{i,n})\Big|\\
    &\leq\Big|\tfrac{C_N}{|x-z_{i,n}-\hat\epsilon_n z|^{N-2}}-\tfrac{C_N}{|x-z_{i,n}|^{N-2}}\Big|+2L\delta_n^{N-2} \\[0.5ex]
    &=C_N\Big|\tfrac{|x-z_{i,n}|^{N-2}-|x-z_{i,n}-\hat\epsilon_n z|^{N-2}}{|x-z_{i,n}-\hat\epsilon_n z|^{N-2}|x-z_{i,n}|^{N-2}}\Big|+2L\delta_n^{N-2}\\[0.5ex]
    &\leq C(r,R,N)(\hat\epsilon_n+\delta_n^{N-2}).
\end{align*}
In particular, $\abs{(A_n)}$ converges uniformly to $0$ as $n\to+\infty$.
Concerning $(B_n)$ we use \eqref{Robinscaled} to deduce that
\begin{align*}
(B_n)&=\int_{B_{2R_0}(0)} G_n(x,\hat\epsilon_n z+z_{i,n})\big|[w_{i,n}-1]_+^p-[w_0(z)-1]_+^p\big|\,dz \\
&\leq\int_{B_{2R_0}}(\tfrac{C}{r^{N-2}}+L\delta_n^{N-2})\big|[w_{i,n}-1]_+^p-[w_0(z)-1]_+^p\big|\,dz \\[0.5ex]
&\leq C(r,R,R_0,N)\|[w_{i,n}-1]_+^p-[w_0(z)-1]_+^p\|_{L^{\infty}(B_{2R_0})} \longrightarrow0 \quad \text{as $n\to\infty$}.
\end{align*}
Hence, this shows that 
$$|\alpha_n(x)|\to0, \,\,\text{as}\,\, n\to+\infty, \,\,\text{uniformly in $\bar{B}_{2R}\backslash \widehat{\Sigma}_{r}$}.$$
Now, we want to estimate \eqref{II}. Let $i\in\{k+1,\dots,l+\bar m\}$, then
\begin{align*}
    \int_{B_{2R_0\hat\epsilon_n}(z_{i,n})} G_n(x,y)\hat\mu_n^{\frac{N}{2}}[\hat v_n(y)-1]_+^p\,dy&=\int_{B_{2R_0}(0)} G_n(x,\hat\epsilon_n z+z_{i,n})[\hat v_n(\hat\epsilon_n z+z_{i,n})-1]_+^p\,dz\\
    &=\int_{B_{2R_0}(0)} G_n(x,\hat\epsilon_n z+z_{i,n})[w_{i,n}(z)-1]_+^p\,dz,
\end{align*}
where $w_{i,n}\to w_0$ in $C^2_{loc}$ as in definition \ref{spikessequencesdef} (iv). However, in this case, by using the fact that $|z_{i,n}|\to+\infty$, $|x|\leq 2R$, $|z|\leq2R_0$ and \eqref{Hndefinition}, we deduce that
\[
G_n(x,\hat\epsilon_n z+z_{i,n})=\frac{C_N}{|x-\hat\epsilon_n z-z_{i,n}|^{N-2}}+H_n(x,\hat\epsilon_n z+z_{i,n})\to 0,
\]
as $n\to+\infty$, uniformly in $x\in \bar B_{2R}$. Hence, 
\begin{align*}
    \sup_{x\in \overline{B}_{2R}} \bigg|\int_{B_{2R_0\hat\epsilon_n}(z_{i,n})} G_n(x,y)\hat\mu_n^{\frac{N}{2}}[\hat v_n(y)-1]_+^p\,dy\bigg|=o_n(1),
\end{align*}
for every $i\in\{k+1,\dots,l+\bar m\}$, therefore 
\[
\abs{\beta_n(x)}=o_n(1)\to0, \,\,\text{as}\,\, n\to+\infty,\,\, \text{uniformly in $\bar{B}_{2R}$}.
\]
At this point, from \eqref{expression}, \eqref{tripartition} and the above estimates, we are able to deduce that
\begin{equation*}
|R_n(x)|:=\left|\hat\epsilon_n^{N-2}\hat v_n(x)-\sum_{i=1}^k M_{p,0}\tfrac{C_N}{|x-z_{i,n}|^{N-2}}\right|\longrightarrow0,\,\,\text{as}\,\, n\to+\infty,
\end{equation*}
uniformly in $x\in\bar B_{2R}\backslash \widehat\Sigma_r$, which is our first claim. 

Now, let us consider the first derivatives of $R_n$. Namely, let us fix $j\in\{1,\dots,k\}$ and consider 
\begin{equation*}
\partial_{x_j}R_n=\partial_{x_j}\left(\hat\epsilon_n^{N-2}\hat v_n(x)-\sum_{i=1}^k M_{p,0}\tfrac{C_N}{|x-z_{i,n}|^{N-2}}\right).
\end{equation*}
By using \eqref{expression} and \eqref{tripartition}, we deduce that
\begin{align*}
\partial_{x_j}R_n&=\partial_{x_j}\left(\hat\epsilon_n^{N-2}\hat v_n(x)-\sum_{i=1}^k M_{p,0}\tfrac{C_N}{|x-z_{i,n}|^{N-2}}\right)\\
&=\partial_{x_j}\left(\int_{\O_n}G_n(x,y)\hat\mu_n^{\frac{N}{2}}[\hat v_n(y)-1]_+^p\,dy-\sum_{i=1}^k M_{p,0}G_n(x,z_{i,n})+\sum_{i=1}^{k} M_{p,0}H_n(x,z_{i,n}) \right)\\
&=\partial_{x_j}\alpha_n(x)+\partial_{x_j}\beta_n(x)+\partial_{x_j}\gamma_n(x). 
\end{align*}
\noindent Since $\mathrm{dist}(B_{2R_0\hat\epsilon_n }(z_{i,n}),\bar B_R\backslash\widehat\Sigma_r)\geq \frac{r}{2}$, for $n$ large enough,
we can differentiate under the integral sign and obtain that 
\begin{equation}
    \label{I2}
\partial_{x_j}\alpha_n(x)=\sum_{i=1}^k\bigg(\int_{B_{2R_0\hat\epsilon_n}(z_{i,n})} \partial_{x_j}G_n(x,y)\hat\mu_n^{\frac{N}{2}}[\hat v_n(y)-1]_+^p\,dy \,\,-M_{p,0} \partial_{x_j}G_n(x,z_{i,n})\bigg),
\end{equation}
\begin{equation*}
\partial_{x_j}\beta_n(x)=\sum_{i=k+1}^{l+\bar m}\int_{B_{2R_0\hat\epsilon_n}(z_{i,n})} \partial_{x_j}G_n(x,y)\hat\mu_n^{\frac{N}{2}}[\hat v_n(y)-1]_+^p\,dy
\end{equation*}
and
\begin{equation*}
\partial_{x_j}\gamma_n(x)=\sum_{i=1}^{k} M_{p,0}\partial_{x_j}H_n(x,z_{i,n}).
\end{equation*}
Now, we argue as for \eqref{I} in order to estimate \eqref{I2}. For any $i\in\{1,\dots,k\}$, we see that
\begin{align*}
    \int_{B_{2R_0\hat\epsilon_n}(z_{i,n})} &\partial_{x_j}G_n(x,y)\hat\mu_n^{\frac{N}{2}}[\hat v_n(y)-1]_+^p\,dy \,-M_{p,0} \partial_{x_j}G_n(x,z_{i,n})\\
    &=\int_{B_{2R_0}(0)} \partial_{x_j}G_n(x,\hat\epsilon_n z+z_{i,n})[\hat v_n(\hat\epsilon_n z+z_{i,n})-1]_+^p\,dz \,-M_{p,0} \partial_{x_j}G_n(x,z_{i,n})\\
    &=\int_{B_{2R_0}(0)} \partial_{x_j}G_n(x,\hat\epsilon_n z+z_{i,n})[w_{i,n}(z)-1]_+^p\,dz \,-M_{p,0} \partial_{x_j}G_n(x,z_{i,n}),
\end{align*}
where $w_{i,n}\to w_0$ in $C^{2}_{loc}$ as in definition \ref{spikessequencesdef} (iv).
Thus, by using the definition of $M_{p,0}$ \eqref{eq:M_p,0-def}, we have that
\begin{align*}
    \bigg|\int_{B_{2R_0}(0)} &\partial_{x_j}G_n(x,\hat\epsilon_n z+z_{i,n})[w_{i,n}(z)-1]_+^p\,dz \,-M_{p,0} \partial_{x_j}G_n(x,z_{i,n})\bigg|\\
    &\leq \underbrace{\int_{B_{2R_0}(0)} \big|\partial_{x_j}G_n(x,\hat\epsilon_n z+z_{i,n})-\partial_{x_j}G_n(x,z_{i,n})\big|[w_0(z)-1]_+^p\,dz}_{(A_n')} +\\
    &\hspace{2cm}+\underbrace{\int_{B_{2R_0}(0)} \partial_{x_j}G_n(x,\hat\epsilon_n z+z_{i,n})\big|[w_{i,n}-1]_+^p-[w_0(z)-1]_+^p\big|\,dz}_{(B_n')}.\\
\end{align*}
To estimate $(A_n')$, we remind that $|z|\leq 2R_0$ and $x\in B_{2R}\backslash\widehat\Sigma_r$. Then, for $n$ large enough, $|x-z_{i,n}|\geq\tfrac{r}{2}$, for every $i\in\{1,\dots,k\}$ and, by using \eqref{Robinscaledderivative}, we deduce that
\begin{align*}
    \big|\partial_{x_j}G_n(x,&\hat\epsilon_n z+z_{i,n})-\partial_{x_j}G_n(x,z_{i,n})\big|=\\[0.5ex]
    &=\big|\tfrac{C_N(2-N)(x_j-(z_{i,n})_j-\hat\epsilon_n z_j)}{|x-z_{i,n}-\hat\epsilon_n z|^{N}}-\tfrac{C_N(2-N)(x_j-(z_{i,n})_j)}{|x-z_{i,n}|^{N}}+ \partial_{x_j}H_n(x,\hat\epsilon_n z+z_{i,n})-\partial_{x_j}H_n(x,z_{i,n})\big|\\[0.5ex]
    &\leq\big|\tfrac{C_N(2-N)(x_j-(z_{i,n})_j-\hat\epsilon_n z_j)}{|x-z_{i,n}-\hat\epsilon_n z|^{N}}-\tfrac{C_N(2-N)(x_j-(z_{i,n})_j)}{|x-z_{i,n}|^{N}}\big|+2L\delta_n^{N-1} \\[0.5ex]
    &=C_N(N-2)\big|\tfrac{|x-z_{i,n}|^{N}(x-z_{i,n}-\hat\epsilon_n z)_j-|x-z_{i,n}-\hat\epsilon_n z|^{N}(x-z_{i,n})_j}{|x-z_{i,n}-\hat\epsilon_n z|^{N}|x-z_{i,n}|^{N}}\big|+2L\delta_n^{N-1}\\[0.5ex]
    &\leq C(r,R,N)(\hat\epsilon_n+\delta_n^{N-1}).
\end{align*}
In particular, $\abs{(A_n')}\to0$ as $n\to+\infty$ uniformly in $x\in \bar{B}_{2R}\backslash\widehat{\Sigma}_r$.
Concerning $(B_n')$, by using \eqref{Robinscaled}, we deduce that
\begin{align*}
(B_n')&=\int_{B_{2R_0}(0)} \partial_{x_j}G_n(x,\hat\epsilon_n z+z_{i,n})\big|[w_{i,n}-1]_+^p-[w_0(z)-1]_+^p\big|\,dz \\
&\leq\int_{B_{2R_0}}(\tfrac{C}{r^{N-1}}+C(r)\delta_n^{N-1})\big|[w_{i,n}-1]_+^p-[w_0(z)-1]_+^p\big|\,dz \\
&\leq C(r,R,R_0,N) \|[w_{i,n}-1]_+^p-[w_0(z)-1]_+^p\|_{L^{\infty}(B_{2R_0})}\longrightarrow0 \quad \text{as $n\to+\infty$}.
\end{align*}
Hence, this shows that 
$$|\partial_{x_j}\alpha_n(x)|\to0,\,\,\text{as}\,\, n\to+\infty,$$
uniformly in $x\in\bar B_{2R}\backslash \widehat\Sigma_r$. \\
Concerning the estimate on $\partial_{x_j}\beta_n(x)$, we can argue as in \eqref{II} and have that
$$|\partial_{x_j}\beta_n(x)|\to 0, \,\,\text{as}\,\, n\to+\infty,$$
uniformly in $x\in\bar B_{2R}\backslash \widehat\Sigma_r$. Moreover, regarding \eqref{III}, we use again \eqref{Robinscaledderivative} and deduce that
$$|\partial_{x_j}\gamma_n(x)|\to0, \,\,\text{as}\,\, n\to+\infty,$$
uniformly in $x\in\bar B_{2R}\backslash \widehat\Sigma_r$. As a consequence, we get that
\begin{equation*}
|\partial_{x_j}R_n(x)|:=\left|\partial_{x_j}\left(\hat\epsilon_n^{N-2}\hat v_n(x)-\sum_{i=1}^k M_{p,0}\tfrac{C_N}{|x-z_{i,n}|^{N-2}}\right)\right|\to 0\,\,\,\text{as}\,\, n\to+\infty,
\end{equation*}
uniformly in $x\in\bar B_{2R}\backslash \widehat\Sigma_r$, which concludes the proof of the lemma.
\end{proof}

\noindent

\subsection{Pohozaev identity for interior spike points}\label{subsec:simplepohozaev}

In this subsection we prove, by the means of a Pohozaev-type identity, that the spike points inside $B_{2R}$ are critical points of a suitable Hamiltonian. \\
Let us set $\hat u_n:=\hat\epsilon_n^{2-N}\hat v_n$ in $B_{2R}$; then
\begin{equation}
 \label{equationhatun}
 -\D \hat u_n=\hat\mu_n^{\frac{N}{2}}[\hat v_n-1]_+^p\,\,\,\,\text{in $B_{2R}$},
\end{equation}
 and, by Lemma \ref{lemmaasymptotic}, we deduce that
\begin{equation}
 \label{convergencehatun}
 \hat u_n(x)\to \hat G(x):=M_{p,0}\sum_{i=1}^k  \tfrac{C_N}{|x-z_{i}|^{N-2}},\,\,\,\,\text{in}\,\, C^0_{loc}(\bar B_{R}\backslash\widehat \Sigma_r)
\end{equation}
and
\begin{equation}
 \label{convergencederivativehatun}
 \nabla\hat u_n(x)\to  \nabla \hat G(x)=\sum_{i=1}^kM_{p,0} \tfrac{C_N(2-N)(x-z_{i})}{|x-z_{i}|^{N}},\,\,\,\,\text{in}\,\, C^0_{loc}(\bar B_{R}\backslash\widehat \Sigma_r).
\end{equation}
Now we test the above equation \eqref{equationhatun} against $\nabla\hat u_n$ and obtain the following vectorial Pohozaev-type identity: 
\begin{equation}
 \label{pohozaev}
 \int_{\p \hat\O}\left(-(\p_{\nu}\hat u_n)\nabla\hat u_n+\tfrac{1}{2}|\nabla\hat u_n|^2\nu\right)\,d\sigma=\int_{\p \hat\O}\tfrac{\hat\mu_n^{N-1}}{p+1}[\hat v_n-1]_+^{p+1}\nu\,d\sigma,
\end{equation}
for every open subset $\hat\O\Subset B_{R}$. 
At this point, we fix $j\in\{1, \dots, k\}$ and consider $z_j\in\widehat\Sigma$ and $\hat\O=B_r(z_j)$. Hence, by \eqref{convergencehatun} and \eqref{convergencederivativehatun}, we derive that the left-hand side of \eqref{pohozaev} satisfies, as $n\to+\infty$,
\[
 \int_{\p B_r(z_j)}\left(-(\p_{\nu}\hat u_n)\nabla\hat u_n+\tfrac{1}{2}|\nabla\hat u_n|^2\nu\right)\,d\sigma\to \int_{\p B_r(z_j)}\left(-(\p_{\nu}\hat G)\nabla\hat G+\tfrac{1}{2}|\nabla\hat G|^2\nu\right)\,d\sigma,
\]
while, reminding \eqref{vanishingoutsidespikes}, the right-hand side of \eqref{pohozaev} vanishes for $n$ large enough. 
Hence, we have the following identity:
\begin{equation*}
 \int_{\p B_r(z_j)}\left(-(\p_{\nu}\hat G)\nabla\hat G+\tfrac{1}{2}|\nabla\hat G|^2\nu\right)\,d\sigma=0.
\end{equation*}
Let $r(x)=|x-z_j|$; by definition of $\hat{G}$, one has
\begin{equation*}
 \hat G(x)=M_{p,0}C_N\Big(\frac{1}{r^{N-2}}+F_j(x)\Big),
\end{equation*}
where 
\begin{equation*}
F_j(x):=\sum_{i=1,i\neq j}^{k} \tfrac{1}{|x-z_i|^{N-2}}.
\end{equation*}
At this point, a straightforward calculation shows that
\begin{align*}
 -(\p_{\nu}\hat G)\nabla\hat G&+\tfrac{1}{2}|\nabla\hat G|^2\nu= \\
 &=M_{p,0}^2C_N^2\left[-\tfrac{1}{2}\tfrac{(2-N)^2}{r^{2N-2}}\nabla r+\tfrac{1}{2}|\nabla F_j|^2\nabla r+ \tfrac{N-2}{r^{N-1}} \nabla F_j-(\nabla F_j\cdot\nabla r)\nabla F_j\right],
\end{align*}
which implies 
\begin{equation}
 \label{limitpohozaev2}
\int_{\p B_r(z_j)} \left[-\tfrac{1}{2}\tfrac{(2-N)^2}{r^{2N-2}}\nabla r+\tfrac{1}{2}|\nabla F_j|^2\nabla r+ \tfrac{N-2}{r^{N-1}} \nabla F_j-(\nabla F_j\cdot\nabla r)\nabla F_j\right]\,d\sigma=0.
\end{equation}
By inspection, we conclude that the first integral is $0$ by symmetry, while all the components of the vectors
\[
 \int_{\p B_r(z_j)} \tfrac{1}{2}|\nabla F_j|^2\nabla r\,d\sigma\,\,\,\,\,\text{and}\,\,\, \int_{\p B_r(z_j)}(\nabla F_j\cdot\nabla r)\nabla F_j\,d\sigma
\]
are $O(r^{N-1})$ as $r\to0^+$ (notice that $F_j$ is uniformly bounded in $C^1(B_{2r})$ by definition). Finally,
\[
 \int_{\p B_r(z_j)} \tfrac{N-2}{r^{N-1}} \nabla F_j\,d\sigma=(N-2)\abs{\Sf^{N-1}}\fint_{\p B_r(z_j)} \nabla F_j\,d\sigma.
\]
As a consequence, by letting $r\to 0$ in \eqref{limitpohozaev2}, we deduce that
\begin{equation*}
 \label{finalequation}
 \nabla F_j(z_j)=(2-N)\sum_{i=1,i\not=j}^k\frac{z_j-z_i}{\abs{z_j-z_i}^{N}}=0.
\end{equation*}
The above calculations holds for every $j\in\{1,\dots,k\}$, therefore $z_1,\dots,z_k$ satisfy the following system:
\begin{equation}
 \label{Finalequation}
 \sum_{i=1,i\neq j}^{k} \tfrac{(z_i-z_j)}{|z_i-z_j|^N}=0 \qquad \text{for every $j\in\{1,\dots,k\}$}.
\end{equation}
Since by construction $z_1=0$ and $|z_2|=1$, then by Lemma \ref{lemmafinal}, we get a contradiction. This concludes the proof of Theorem \ref{Simplicitytheorem}.

\medskip

Here we prove the following technical lemma:
\begin{lemma}\label{lemmafinal}
  Let $k\geq 2$ and consider $\{z_i\}_{i=1}^k\subset\R^N$ such that $z_i\neq z_j$, for every $i\neq j$. Then system \eqref{Finalequation} has no solution.

\end{lemma}
\begin{proof}
 The proof for $k=2$ is trivial. \\
 Let $k\geq 3$. Let us assume \eqref{Finalequation} hold and let $z_{i_1}\in\{z_1,\dots,z_k\}$ be such that
 \[
  (z_{i_1})_1=\max_{i\in\{1,\dots,k\}} (z_i)_1,
 \]
where $(v)_1$ is the first component of the vector $v\in\R^N$. \\
If there exists $\hat z\in\{z_1,\dots,z_k\}$ such that $(\hat z)_1<(z_{i_1})_1$, then 
\[
 \bigg(\sum_{i=1,i\neq i_1}^{k} \frac{(z_i-z_{i_1})}{|z_i-z_{i_1}|^N}\bigg)_1<0.
\]
In particular, this implies that
\[
  \sum_{i=1,i\neq i_1}^{k} \frac{(z_i-z_j)}{|z_i-z_j|^N}\neq0,
 \]
contradicting \eqref{Finalequation}. Hence, we must necessarily consider all the first components of $\{z_1,\dots, z_k\}$
equal, namely $$(z_1)_1=\dots=(z_k)_1.$$ On the other hand, we can repeat the same argument for all the other components and deduce that necessarily
\[
 z_1=\dots=z_k,
\]
which is a contradiction.
\end{proof}

\section{Nonexistence of boundary spike points}\label{sec:Boundaryspikes}

Let us prove Theorem \ref{Boundaryspikestheorem}. As for the case of interior spikes, $v_n$ is a spike sequence, as a consequence of Theorem \ref{BJWtheorem}. Again, we argue by contradiction and assume that there exists at least a boundary spike point $\tilde z\in\partial \O\cap\Sigma$. Then, we suppose there exist $x_{1,n},\dots,x_{l,n}$, with $l\geq1$, sequences of local maxima corresponding to $l$ spike profiles clustering at $\tilde z$.\\
Let us define, without loss of generality, 
\[
d_n:=\underset{i\in\{1,\dots,l\}}\min \mathrm{dist}(x_{i,n},\partial\O)=\mathrm{dist}(x_{1,n},\partial\O)=|x_{1,n}-\tilde z_n|.
\]
where $\tilde z_n\in\partial \O$, $\tilde z_n\to\tilde z$, as $n\to+\infty$. Let us define 
\begin{align*}
\tilde v_n(x):=v_n(\tilde z_n+d_n x),
\end{align*}
and let $\tilde\mu_n:=d_n^2\mu_n$. Then $\tilde v_n$ satisfies
\begin{equation*}
    \begin{cases}
    -\D \tilde v_n=\tilde\mu_n[\tilde v_n-1]_+^p\,\, &\text{in}\, \O_n, \\[0.5ex]
    \tilde v_n=0 \,\, &\text{on}\, \p\O_n,  
\end{cases}
\end{equation*}
where $\O_n=d_n^{-1}(\O-\tilde z_n)$. Without loss of generality, up to a traslation and a rotation, we can assume that, as $n\to\infty$, $\O_n$ converges in $C^2_{loc}$ to $\R_-^N:=\{x=(x',x_N)\in\R^N\mid x_N<0\}$. 

Also, by reminding that $\epsilon_n^2=\mu_n^{-1}$ and the fact that $v_n$ is a spike sequence according to Theorem \ref{BJWtheorem}, we know that $(iii)$ of Definition \ref{spikessequencesdef} holds. Hence, by the very definition of spike set (see Definition \ref{def:spikeset}), we deduce that $\tilde {\mu}_n\to+\infty$. Moreover, assumption \eqref{HypothesisAintroduction} implies by a change of variable that
\[
\tilde {\mu}_n^{\frac{N}{2}}\int_{\O_n}[\tilde v_n-1]_+^p\,dx\leq C_0.
\]
We notice also that, by denoting with 
\[
z_{i,n}:=d_n^{-1}(x_{i,n}-\tilde z_n), \,\,\,\,\,\text{for}\,\,i\in\{1,\dots,l\},
\]
In particular, by construction, $z_{1,n}\to z_1=(0,\dots,0,-1)$ as $n\to\infty$. Moreover, we can assume that $z_{1,n},\dots,z_{k,n}$, with $k\leq l$, are the only spike maxima satisfying
\[
\limsup_n \mathrm{dist}(z_{i,n},(0,\dots,0))<+\infty,\,\,\,\,\, \forall i\in\{1,\dots,l\}.
\]
Without loss of generality, let us define 
\[
\delta_n:=\underset{{i,j\in\{1,\dots,k\}}, i\neq j}{\min}|z_{i,n}-z_{j,n}|=|z_{1,n}-z_{2,n}|.
\]
Then, we are required to study two different cases: 
\begin{itemize}
    \item (Case 1)  $\delta_n\to0^+$ as $n\to+\infty$;
    \item (Case 2)   $\delta_n\geq c>0$ $\forall n\in \N$.
\end{itemize}
Concerning Case 1, we can define $\hat v_n$ as
\[
\hat v_n(x):=\tilde v_n(z_{1,n}+\delta_nx) 
\]
for $x\in\hat\O_n:=\delta_n^{-1}(\O-z_{1,n})$. Here, we are in a situation similar to the one discussed before about the interior spikes points, so the contradiction follows from a step by step adaptation of the arguments employed in Section \ref{sec:interiorspikes}.\\
Hence, we are reduced to consider Case 2. If $\delta_n\to+\infty$ as $n\to+\infty$, then $k=1$, while if $\delta_n\to L\in(0,+\infty)$ as $n\to+\infty$, we deduce $k>1$, which means that $\exists R>0$ such that $\{z_{i,n}\}_{i=1}^k\subset B_R(0)\cap \O_n$ for $n$ large enough.   \\
Moreover, for any $i\in\{1,\dots, k\}$, we assume that there exist $z_i\in B_R^-(0):=\{x\in B_R(0) \,| \,x_N<0\}$ such that, up to a subsequence, $z_{i,n}\to z_i$, as $n\to+\infty$. We also remark that, by construction, $(z_i)_N\leq-1$ for every $i\in\{1,\dots,k\}$. \\
Let us define $\widetilde \Sigma:=\{z_1,\dots, z_k\}$ and $\widetilde \Sigma_r:=\bigcup_{i=1}^k B_r(z_i),$  for $r>0$ small. As for the case of interior spikes, we now proceed in two steps:
    \begin{itemize}
        \item [(1)] First, we derive pointwise estimates for $\tilde v_n$ and its derivatives in $(B_R\cap {\Omega}_n)\backslash\widetilde{\Sigma}_r$;
        \item [(2)] Secondly, we use a Pohozaev-type identity to get a contradiction.
    \end{itemize}
\medskip

\subsection{Pointwise estimates for $\tilde{v}_n$}

Let us consider $G_n(x,y):=G_{\O_n}(x,y)$, the Green's function of the Laplacian in $\O_n$. It is straightforward to verify that 
$$
G_n(x,y)=d_n^{N-2} G_{\O}(\tilde z_n+d_n x,\tilde z_n+d_n y)= \tfrac{C_N}{|x-y|^{N-2}}+ H_n(x,y),
$$
where $G_\O$ is the Green's function in $\O$ and
\begin{equation}
    \label{Hndefinition2}
H_n(x,y):=d_n^{N-2} H_{\O}(\tilde z_n+d_n x,\tilde z_n+d_n y)
\end{equation}
is the regular part of the Green function in $\O_n$.\\
Let $U_R:=\bar B_{2R}(0)\cap\{y\mid y_N<-\tfrac{1}{2}\}$. From what we said, $U_R\Subset \Omega_n$ for $n$ large enough.
We now claim the following important property for $H_n$:
\begin{lemma}
    \label{Robinbounded}
   $H_n(x,y)$ is uniformly bounded in $C^2(\O_n\times U_R)$, provided $n$ large enough.
\end{lemma}
Before proving this result we need the following lemma:
\begin{lemma}
    \label{Hproperty}
Let $d_\O(x):=\mathrm{dist}(x,\partial \O)$. Then, there exists $C=C(\O)>0$ such that, for every $x,y\in\O$,
    \begin{equation*}
        |H_\O(x,y)|\leq\tfrac{C}{d_\O(x)^{N-2}},
    \end{equation*}
    \begin{equation*}
|\nabla_xH_\O(x,y)|\leq\tfrac{C}{d_\O(x)^{N-1}},
    \end{equation*}
    \begin{equation*}
|\nabla_x^2H_\O(x,y)|\leq\tfrac{C}{d_\O(x)^{N}}.
    \end{equation*}
\end{lemma}

\begin{proof}
We fix $x\in\O$ and notice that, by definition of $H_\O$, we have that, for every $y\in\partial \O$,
\[
|H_\O(x,y)|=\tfrac{C_N}{|x-y|^{N-2}}\leq \tfrac{C_N}{d_\O(x)^{N-2}}.
\]
Then, by applying the maximum principle, we get that
\[
|H_\O(x,y)|\leq \tfrac{C_N}{d_\O(x)^{N-2}}
\]
for every $y\in\O$. The other inequalities are proved by arguing in the same way.
\end{proof}

\begin{proof}[Proof of Lemma \ref{Robinbounded}]
By definition,
\[
H_n(x,y)=d_n^{N-2} H_{\O}(\tilde z_n+d_n x,\tilde z_n+d_n y),
\]
\[
\partial_{x_i}H_n(x,y)=d_n^{N-1} \partial_{x_i}H_{\O}(\tilde z_n+d_n x,\tilde z_n+d_n y),
\]
\[
\partial^2_{x_i,x_j}H_n(x,y)=d_n^{N} \partial^2_{x_i,x_j}H_{\O}(\tilde z_n+d_n x,\tilde z_n+d_n y).
\]
Hence, by using Lemma \ref{Hproperty} and the symmetry of $H_n$ in $x,y$, we deduce that
\[
|H_n(x,y)|\leq d_n^{N-2}\tfrac{C}{d_\O(z_n+d_ny)^{N-2}}.
\]
Now, for $n$ large enough, since $y_n<-\tfrac{1}{2}$ and $\partial \O_n$ is approaching $\{y_n=0\}$ near $0$, we can assume without loss of generality that $\mathrm{dist}(y,\partial \O_n)>\tfrac{1}{4}$. Then, this in turn reads as $d_\O(z_n+d_ny)>\tfrac{d_n}{4}$, so from above we have that
\[
|H_n(x,y)|\leq C.
\]
Similarly, by using Lemma \ref{Hproperty} as above, we obtain that
\[
|\partial_{x_i}H_n(x,y)|\leq d_n^{N-1}\tfrac{C}{d_\O(z_n+d_ny)^{N-1}}\leq C
\]
and
\[
|\partial^2_{x_i,x_j}H_n(x,y)|\leq d_n^{N}\tfrac{C}{d_\O(z_n+d_ny)^{N}}\leq C.
\]
This concludes the proof.    
\end{proof}

The next lemma shows that, for $n$ large, $H_n(x,y)$ is arbitrarily $C^0$-close to the regular part of the Green's function for the half-plane $\R^N_-$, that is,
\begin{equation*}
    H_{-}(x,y):=-\frac{C_N}{|x-\tilde y|^{N-2}},
\end{equation*}
where $\tilde y=(y_1,\dots,y_{N-1},-y_N)$.
\begin{lemma}\label{approximateRobinlemma}
Let $U_R:=\bar B_{2R}(0)\cap\{y\mid y_N<-\tfrac{1}{2}\}$; then
\begin{equation*}
    \|H_--H_n\|_{C^0(\O_n\times U_R)}=o_n(1),\,\,\,\,\text{as}\,\,n\to+\infty.
\end{equation*}
\end{lemma}
\begin{proof}
By definition, $H_\O(\cdot,y)|_{\partial \O}=-\frac{C_N}{|\cdot- y|^{N-2}}$,  
and, by \eqref{Hndefinition2},
\[
H_n(x,y)=-\frac{C_N}{|x- y|^{N-2}}\,\,\,\,\,\, \forall x\in\partial \O_n, \, \forall y\in \Omega_n.
\]
Fix $\eta>0$ such that $\frac{1}{\eta}>2R$; then there exists $n(\eta)$ such that one has $(\partial \O_n\cap B_{\frac{1}{\eta}}(0))\subset\{x\,|\, |x_N|<\eta\}$ for any $n\geq n(\eta)$.
Now, if $x\in\partial \O_n\cap B_{\frac{1}{\eta}}$ and $y\in U_R$, then
\begin{align*}
    |H_-(x,y)-H_n(x,y)|&=C_N\left|\tfrac{1}{|x-y|^{N-2}}-\tfrac{1}{|x-\tilde y|^{N-2}}\right|\leq C(N,R)\eta
\end{align*}
for $n\geq n(\eta)$, where we used the fact that $|x-y|\geq\frac{1}{4}$.\\
On the other hand, if $x\in\partial \O_n\cap (B_{\frac{1}{\eta}})^c$ and $y\in U_R$,  then $\mathrm{dist}(y,\O_n\cap (B_{\frac{1}{\eta}})^c)>\tfrac{1}{2\eta}$ and $\mathrm{dist}(\tilde y,\O_n\cap (B_{\frac{1}{\eta}})^c)>\tfrac{1}{2\eta}$. This implies that
\begin{align*}
    |H_-(x,y)-H_n(x,y)|&=C_N\left|\tfrac{1}{|x-y|^{N-2}}-\tfrac{1}{|x-\tilde y|^{N-2}}\right|\leq 2C_N(2\eta)^{N-2}.
\end{align*}
Since $\eta>0$ is arbitrary, we conclude that
\[
\|H_-(x,y)-H_n(x,y)\|_{C^0(\partial \O_n\times U_R)}=o_n(1),\,\,\,\,\,\,\text{as}\,\,n\to+\infty.
\]
Let us fix $y\in U_R$ and consider $\Theta_n^y(x):=H_-(x,y)-H_n(x,y)$; then $\Theta_n^y$ is harmonic in $\Omega_n$ and $\sup_{x\in \p\O_n}\abs{\Theta_n^y(x)}=o_n(1)$, as $n\to+\infty$, uniformly in $y\in U_R$. Then, by the maximum principle, we get    
that $\sup_{y\in U_R}\| \Theta_n^y\|_{L^\infty(\O_n)}=o_n(1)$, as $n\to+\infty$; this concludes the proof.
\end{proof}

Combining the results of Lemma \ref{Robinbounded} and Lemma \ref{approximateRobinlemma}, by Ascoli-Arzelà theorem, we have that
\begin{equation}
\label{C1convergence}
\|H_n-H_-\|_{C^1(U_R\times U_R)}=o_n(1),\,\,\,\,\,\text{as}\,\,n\to+\infty.
\end{equation}
 Now, we are ready to derive the expansion of $\tilde v_n$ in $U_R\backslash\widetilde{\Sigma}_r$. We remind that $\{z_i\}_{i=1}^l\subset B_{2R}\cap\{x_N\leq-1\}$, so $\widetilde{\Sigma}_r\Subset U_R$ for $r$ small.

\begin{lemma} \label{approximationlemmaboundary}
Let $\tilde\epsilon_n^2:=\tilde\mu_n^{-1}$ and $x\in U_R\backslash \widetilde\Sigma_r$. Then
\begin{equation*}
    \tilde\epsilon_n^{2-N}\tilde v_n(x)=\sum_{i=1}^kM_{p,0} G_-(x,z_{i,n})+R_n(x),
\end{equation*}
where $\|R_n\|_{C^1(U_R\backslash \widetilde \Sigma_r)}=o_n(1)$, as $n\to+\infty$, and $G_-$ is the Green's function of the Laplacian in the half-space $\{x\in\R^N\,|\, x_N<0\}$, namely
\begin{equation*}
    G_-(x,y)=\frac{C_N}{|x-y|^{N-2}}-\frac{C_N}{|x-\tilde y|^{N-2}},
\end{equation*}
with $\tilde y:=(y_1,\dots,y_{N-1},-y_N)$.  
\end{lemma}

\begin{proof}
By using the Green representation formula for $\tilde v_n$ in $\O_n$ and the definition of $G_n$, we have that
\begin{align}\nonumber
   R_n(x)&:=\tilde\epsilon_n^{2-N}\tilde v_n(x)-\sum_{i=1}^kM_{p,0} G_-(x,z_{i,n}) \\ \nonumber
    &=\tilde\epsilon_n^{2-N}\tilde v_n(x)-\sum_{i=1}^k M_{p,0}G_n(x,z_{i,n})+\sum_{i=1}^{k} M_{p,0}(H_n(x,z_{i,n})-H_-(x,z_{i,n})) \\
    &=\underbrace{\int_{\O_n}G_n(x,y)\tilde\mu_n^{\frac{N}{2}}[\tilde v_n(y)-1]_+^p\,dy-\sum_{i=1}^k M_{p,0}G_n(x,z_{i,n})}_{:=g_n(x)}+\tilde R_n(x).\nonumber
\end{align}
where, by \eqref{C1convergence},
$$
\norm{\tilde R_n}_{C^1(U_R)}:= \norm{\sum_{i=1}^{k} M_{p,0}(H_n(x,z_{i,n})-H_-(x,z_{i,n}))}=o_n(1),
$$ as $n\to+\infty$.
Now, we want to estimate $g_n(x)$ and we can proceed similarly to the case of interior spikes. However, for the sake of completeness, we sketch the main differences.

\noindent By using the fact that $ v_n$ is a spike sequence, one has
\[
[{v}_n-1]_+^p=0 \,\,\,\,\,\,\,\text{in}\,\,\O\backslash\bigcup_{i=1}^{l+\bar m}B_{2\epsilon_nR_0}(x_{i,n}),
\]
where $l$ is the total number of spike sequences that converge to $\tilde z$ and $\bar m\geq0$ counts the number of spike profiles relative to $v_n$ which may converge to other spike points in $\bar \O$. From the latter property, by definition of $\tilde v_n$, it is straightforward to deduce that
\begin{equation}
\label{vanishingoutsidespikesboundary}
[{\tilde v}_n-1]_+^p=0 \,\,\,\,\,\,\,\text{in}\,\,\O\backslash\bigcup_{i=1}^{l+\bar m}B_{2\tilde\epsilon_nR_0}(z_{i,n})
\end{equation}
and then,
\begin{align*}
\int_{\O_n}&G_n(x,y)\tilde\mu_n^{\frac{N}{2}}[\tilde v_n(y)-1]_+^p\,dy= \\
&=\sum_{i=1}^k\int_{B_{2R_0\tilde\epsilon_n(z_{i,n})}}G_n(x,y)\tilde\mu_n^{\frac{N}{2}}[\tilde v_n(y)-1]_+^p\,dy + \sum_{j=k+1}^{l+\bar m}\int_{B_{2R_0\tilde\epsilon_n(z_{j,n})}}G_n(x,y)\tilde\mu_n^{\frac{N}{2}}[\tilde v_n(y)-1]_+^p\,dy.
\end{align*}
We split the calculation of $g_n(x)$ as follows:
\begin{equation}\label{bipartition}
    \int_{\O_n}G_n(x,y)\tilde\mu_n^{\frac{N}{2}}[\tilde v_n(y)-1]_+^p\,dy-\sum_{i=1}^k M_{p,0}G_n(x,z_{i,n})= \alpha_n(x)+\beta_n(x),
\end{equation}
where
\begin{equation}
    \label{Inew}
\alpha_n(x):=\sum_{i=1}^k\bigg(\int_{B_{2R_0\tilde\epsilon_n}(z_{i,n})} G_n(x,y)\tilde\mu_n^{\frac{N}{2}}[\tilde v_n(y)-1]_+^p\,dy \,\,-M_{p,0} G_n(x,z_{i,n})\bigg)
\end{equation}
and
\begin{equation}
    \label{IInew}
\beta_n(x):=\sum_{i=k+1}^{l+\bar m}\int_{B_{2R_0\tilde\epsilon_n}(z_{i,n})} G_n(x,y)\tilde\mu_n^{\frac{N}{2}}[\tilde v_n(y)-1]_+^p\,dy.
\end{equation}
To begin, we estimate \eqref{Inew}. We remind that $\{z_{i,n}\}_{i=1}^k$ are converging to limit points $z_1,\dots,z_k$ inside $U_R$ respectively. For any $i\in\{1,\dots,k\}$, we see that
\begin{align*}
    \int_{B_{2R_0\tilde\epsilon_n}(z_{i,n})} &G_n(x,y)\tilde\mu_n^{\frac{N}{2}}[\tilde v_n(y)-1]_+^p\,dy \,-M_{p,0} G_n(x,z_{i,n})\\
    &=\int_{B_{2R_0}(0)} G_n(x,\tilde\epsilon_n z+z_{i,n})[\tilde v_n(\tilde\epsilon_n z+z_{i,n})-1]_+^p\,dz \,-M_{p,0} G_n(x,z_{i,n})\\
    &=:\int_{B_{2R_0}(0)} G_n(x,\tilde\epsilon_n z+z_{i,n})[w_{i,n}(z)-1]_+^p\,dz \,-M_{p,0} G_n(x,z_{i,n}),
\end{align*}
where $w_{i,n}\to w_0$ in $C^{2}_{loc}$ as in definition \ref{spikessequencesdef} (iv). Here, we have used the fact that
\begin{equation*}
\tilde v_n(\tilde\epsilon_n z+z_{i,n})=v_n(x_{1,n}+d_n(\tilde\epsilon_n z+z_{i,n}))=v_n(x_{i,n}+\epsilon_n z)=w_{i,n}(z).
\end{equation*}
Thus, by using the definition of $M_{p,0}$ \eqref{eq:M_p,0-def}, we have that
\begin{align*}
    \bigg|\int_{B_{2R_0}(0)} &G_n(x,\tilde\epsilon_n z+z_{i,n})[w_{i,n}(z)-1]_+^p\,dz \,-M_{p,0} G_n(x,z_{i,n})\bigg|\\
    &\leq \underbrace{\int_{B_{2R_0}(0)} \big|G_n(x,\hat\epsilon_n z+z_{i,n})-G_n(x,z_{i,n})\big|[w_0(z)-1]_+^p\,dz}_{(\tilde A)} +\\
    &\hspace{2cm}+\underbrace{\int_{B_{2R_0}(0)} G_n(x,\tilde\epsilon_n z+z_{i,n})\big|[w_{i,n}-1]_+^p-[w_0(z)-1]_+^p\big|\,dz}_{(\tilde B)}.
\end{align*}
To estimate $(\tilde A)$, we remind that $|z|\leq 2R_0$ and $x\in U_R\backslash\widetilde\Sigma_r$. Then, for $n$ large enough, $|x-z_{i,n}|\geq\tfrac{r}{2}$, for every $i\in\{1,\dots,k\}$.
Therefore, by using Lemma \ref{Robinbounded}, we deduce that
\begin{align*}
    \Big|G_n(x,\tilde\epsilon_n z&+z_{i,n})-G_n(x,z_{i,n})\Big|=\\[0.5ex]
    &=\Big|\tfrac{C_N}{|x-z_{i,n}-\tilde\epsilon_n z|^{N-2}}-\tfrac{C_N}{|x-z_{i,n}|^{N-2}}+ H_n(x,\tilde\epsilon_n z+z_{i,n})-H_n(x,z_{i,n})\Big|\\[0.5ex]
    &\leq\Big|\tfrac{C_N}{|x-z_{i,n}-\tilde\epsilon_n z|^{N-2}}-\tfrac{C_N}{|x-z_{i,n}|^{N-2}}\Big|+C(R_0)\tilde\epsilon_n \\[0.5ex]
    &\leq C(r,R,R_0,N)(\tilde\epsilon_n).
\end{align*}
In particular, $(\tilde A)=o_n(1)$, as $n\to+\infty$.
Concerning $(\tilde B)$, by using again Lemma \ref{Robinbounded}, we deduce that
\begin{align*}
(\tilde B)&=\int_{B_{2R_0}(0)} G_n(x,\tilde\epsilon_n z+z_{i,n})\big|[w_{i,n}-1]_+^p-[w_0(z)-1]_+^p\big|\,dz \\
&\leq\int_{B_{2R_0}}(\tfrac{C}{r^{N-2}}+C)\big|[w_{i,n}-1]_+^p-[w_0(z)-1]_+^p\big|\,dz \\
&\leq C(r,R_0,N)\|[w_{i,n}-1]_+^p-[w_0(z)-1]_+^p\|_{L^{\infty}(B_{2R_0})}\\
&=o_n(1), \,\,\,\,\,\text{as} \,\,n\to+\infty.
\end{align*}
Hence, this shows that 
$$\sup_{U_R\backslash \tilde\Sigma_r} \abs{\alpha_n(x)}\to0, \,\,\text{as}\,\, n\to+\infty.$$
Now, we want to estimate \eqref{IInew}. Let $i\in\{k+1,\dots,l+\bar m\}$, then
\begin{align*}
    \int_{B_{2R_0\tilde\epsilon_n}(z_{i,n})} G_n(x,y)\tilde\mu_n^{\frac{N}{2}}[\tilde v_n(y)-1]_+^p\,dy&=\int_{B_{2R_0}(0)} G_n(x,\tilde\epsilon_n z+z_{i,n})[\tilde v_n(\tilde\epsilon_n z+z_{i,n})-1]_+^p\,dz\\
    &=\int_{B_{2R_0}(0)} G_n(x,\tilde\epsilon_n z+z_{i,n})[w_{i,n}(z)-1]_+^p\,dz,
\end{align*}
where $w_{i,n}\to w_0$ in $C^0_{loc}$ as in definition \ref{spikessequencesdef} (iv). However, in this case, by using the fact that $|z_{i,n}|\to+\infty$, $|x|\leq R$, $|z|\leq2R_0$ and \eqref{Hndefinition2}, we deduce that
\[
G_n(x,\tilde\epsilon_n z+z_{i,n})=\tfrac{C_N}{|x-\tilde\epsilon_n z-z_{i,n}|^{N-2}}+H_n(x,\tilde\epsilon_n z+z_{i,n})\to 0,
\]
as $n\to+\infty$, uniformly in $x\in U_R$. Hence, 
\begin{align*}
    \int_{B_{2R_0\tilde\epsilon_n}(z_{i,n})} G_n(x,y)\tilde\mu_n^{\frac{N}{2}}[\tilde v_n(y)-1]_+^p\,dy=o_n(1),
\end{align*}
for every $i\in\{k+1,\dots,l+\bar m\}$ and so 
$$
\sup_{U_R} \abs{\beta_n (x)}\to0,\,\,\text{as}\,\,n\to+\infty.
$$
At this point, from \eqref{bipartition}, we are able to deduce that
\begin{equation*}
    \sup_{U_R\backslash \tilde\Sigma_r} |R_n(x)|:=\sup_{U_R\backslash \tilde\Sigma_r} \left|\tilde\epsilon_n^{N-2}\tilde v_n(x)-\sum_{i=1}^k M_{p,0}G_-(x,z_{i,n})\right|\to 0,\,\,\text{as}\,\, n\to+\infty
\end{equation*}
Concerning the derivative of $R_n$ , we can argue as in the interior spikes case with the same minor modifications we did above. This concludes the proof of the lemma.
\end{proof}

\subsection{Pohozaev identity for boundary spike points}\label{subsec:boundarypohozaev}

In this final step, we prove, by the means of a Pohozaev-type identity, that the spike points inside $U_R$ are critical points of a suitable Hamiltonian. \\
Let us set $\tilde u_n:=\tilde\epsilon_n^{2-N}\tilde v_n$ in $U_{R}$, then
\begin{equation}
 \label{equationtildeun}
 -\D \tilde u_n=\tilde\mu_n^{\frac{N}{2}}[\tilde v_n-1]_+^p\,\,\,\,\text{in $U_{R}$},
\end{equation}
 and, by Lemma \ref{approximationlemmaboundary}, we deduce that
\begin{equation*}
 \tilde u_n(x)\to \tilde G(x):=M_{p,0}\sum_{i=1}^k  G_-(x,z_i),\,\,\,\,\text{in}\,\, C^0_{loc}(\bar U_{R}\backslash\widetilde \Sigma_r)
\end{equation*}
and
\begin{equation*}
 \nabla\tilde u_n(x)\to  \nabla \tilde G(x)=\sum_{i=1}^kM_{p,0} \nabla G_-(x,z_i),\,\,\,\,\text{in}\,\, C^0_{loc}(\bar U_{R}\backslash\widetilde \Sigma_r),
\end{equation*}
as $n\to+\infty$.
Now we test the above equation \eqref{equationtildeun} against $\nabla\tilde u_n$ and obtain the following vectorial Pohozaev-type identity, 
\begin{equation}
 \label{pohozaev2}
 \int_{\p \tilde\O}\left(-(\p_{\nu}\tilde u_n)\nabla\tilde u_n+\tfrac{1}{2}|\nabla\tilde u_n|^2\nu\right)=\int_{\p \tilde\O}\tfrac{\tilde\mu_n^{N-1}}{p+1}[\tilde v_n-1]_+^{p+1}\nu,
\end{equation}
for every open subset $\tilde\O\Subset U_{R}$. \\
At this point, we fix $j\in\{1, \dots, k\}$ and consider $z_j\in\widetilde\Sigma$ and $\tilde\O=B_r(z_j)$. Hence, we derive that the left-hand side of \eqref{pohozaev2} satisfies
\[
 \int_{\p B_r(z_j)}\left(-(\p_{\nu}\tilde u_n)\nabla\tilde u_n+\tfrac{1}{2}|\nabla\tilde u_n|^2\nu\right)\,d\sigma\to \int_{\p B_r(z_j)}\left(-(\p_{\nu}\tilde G)\nabla\tilde G+\tfrac{1}{2}|\nabla\tilde G|^2\nu\right)\,d\sigma
\]
as $n\to+\infty$, while, reminding \eqref{vanishingoutsidespikesboundary}, the right-hand side of \eqref{pohozaev2} vanishes, 
\[
 \int_{\p B_r(z_j)}\tfrac{\tilde\mu_n^{N-1}}{p+1}[\tilde v_n-1]_+^{p+1}\nu=0,
\]
for $n$ large enough. 
Hence, we have the following identity:
\begin{equation}
 \label{limitpohozaevboundary}
 \int_{\p B_r(z_j)}\left(-(\p_{\nu}\tilde G)\nabla\tilde G+\tfrac{1}{2}|\nabla\tilde G|^2\nu\right)\,d\sigma=0.
\end{equation}
Now, by definition of $\tilde G$, by denoting $r(x)=|x-z_j|$, then we can write
\begin{align*}
 \tilde G(x)=M_{p,0} C_N\left(\tfrac{1}{r^{N-2}}+F_j(x)\right),
\end{align*}
where 
\begin{align*}
F_j(x):=- \tfrac{1}{|x-\tilde z_j|^{N-2}}+\sum_{i=1,i\neq j}^{k} \left(\tfrac{1}{|x-z_i|^{N-2}} -\tfrac{1}{|x-\tilde z_i|^{N-2}}\right).
\end{align*}
Then, by a simple computation, we have that
\begin{align*}
 -(\p_{\nu}\tilde G)\nabla\tilde G&+\tfrac{1}{2}|\nabla\tilde G|^2\nu= \\
 &=M_{p,0}^2C_N^2\left[-\tfrac{1}{2}\tfrac{(2-N)^2}{r^{2N-2}}\nabla r+\tfrac{1}{2}|\nabla F_j|^2\nabla r+ \tfrac{N-2}{r^{N-1}} \nabla F_j-(\nabla F_j\cdot\nabla r)\nabla F_j\right],
\end{align*}
which, by \eqref{limitpohozaevboundary}, implies that
\begin{equation}
 \label{limitpohozaevboundary2}
\int_{\p B_r(z_j)} \left[-\tfrac{1}{2}\tfrac{(2-N)^2}{r^{2N-2}}\nabla r+\tfrac{1}{2}|\nabla F_j|^2\nabla r+ \tfrac{N-2}{r^{N-1}} \nabla F_j-(\nabla F_j\cdot\nabla r)\nabla F_j\right]\,d\sigma=0.
\end{equation}
By inspection, we conclude that the first integral is $0$ by symmetry, while all the components of the vectors
\[
 \int_{\p B_r(z_j)} \tfrac{1}{2}|\nabla F_j|^2\nabla r\,d\sigma\,\,\,\,\,\text{and}\,\,\, \int_{\p B_r(z_j)}(\nabla F_j\cdot\nabla r)\nabla F_j\,d\sigma
\]
are $O(r^{N-1})$, as $r\to0$. Finally,
\[
 \int_{\p B_r(z_j)} \tfrac{1}{r^{N-1}} \nabla F_j\,d\sigma=\abs{\Sf^{N-1}}\fint_{\p B_r(z_j)} \nabla F_j\,d\sigma,
\]
so, by letting $r\to 0$ in \eqref{limitpohozaevboundary2}, we deduce that
\begin{equation*}
 \nabla F_j(z_j)=0.
\end{equation*}
Since we can repeat the same procedure for every $j\in\{1,\dots,k\}$, we deduce from the definition of $F_j$ that
\begin{equation*}
  -\frac{(z_j-\tilde z_j)}{|z_j-\tilde z_j|^N}+\sum_{i=1,i\neq j}^{k} \left(\frac{(z_j-z_i)}{|z_j-z_i|^N}-\frac{(z_j-\tilde z_i)}{|z_j-\tilde z_i|^N}\right)=0, \qquad \text{for every $j\in\{1,\dots,k\}$},
\end{equation*}
or, equivalently,
\begin{equation}
 \label{Finalequation2B}
  \frac{(\tilde z_j-z_j)}{|\tilde z_j-z_j|^N}+\sum_{i=1,i\neq j}^{k} \left(\frac{(\tilde z_i-z_j)}{|\tilde z_i-z_j|^N}-\frac{(z_i- z_j)}{|z_i- z_j|^N}\right)=0, \qquad \text{for every $j\in\{1,\dots,k\}$}.
\end{equation}
Let us consider the $N^{\mathrm{th}}$ component of \eqref{Finalequation2B} and let $j\in\{1,\dots,k\}$ be such that $(\tilde z_j)_N=\underset{i=1,\dots,k}\min (\tilde z_i)_N$ (we remark that $(\tilde z_i)_N>0$, for every $i\in\{1,\dots,k\}$). Now, by looking at \eqref{Finalequation2B} for this particular index $j$, we have that
\begin{align*}
    \frac{(\tilde z_{j}-z_{j})_N}{|\tilde z_{j}-z_{j}|^N}&+\sum_{i=1,i\neq {j}}^{k} \bigg(\frac{(\tilde z_i-z_{j})}{|\tilde z_i-z_{j}|^N}-\frac{(z_i- z_{ j})}{|z_i- z_{ j}|^N}\bigg)_N \\
    &=\underbrace{\frac{2(\tilde z_{j})_N}{|\tilde z_{j}-z_{j}|^N}}_{>0}+\sum_{i=1,i\neq {j}}^{k} \bigg(\underbrace{\frac{(\tilde z_i)_N+(\tilde z_{j})_N}{|\tilde z_i-z_{ j}|^N}}_{>0}+\underbrace{\frac{(z_{j})_N-(z_i)_N}{|z_i- z_{j}|^N}}_{\geq0}\bigg)>0,
\end{align*}
which is in contradiction with \eqref{Finalequation2B}. This proves the theorem.

\section*{Acknowledgements}
The authors are supported by the PRIN Project 2022AKNSE4 {\em Variational and Analytical aspects of Geometric PDE} and are members of GNAMPA, as part of INdAM. The first author is also supported by the MIUR Excellence Department Project MatMod@TOV awarded to the Department of Mathematics, University of Rome Tor Vergata, codice CUP E83C23000330006 and by University of Rome Tor Vergata Project \acc E.P.G.P.'', codice E83C25000550005.\\
We would like to express our sincere gratitude to Prof. Daniele Bartolucci for the inspiring discussions concerning the topics of the paper.


\printbibliography[heading=bibintoc]

\end{document}